\documentclass[a4paper]{amsart}

\usepackage{todonotes}
\usepackage{mathtools}
\usepackage{float}
\usepackage{subcaption}
\usepackage{graphicx}
\graphicspath{{figures/}}

\newtheorem{theorem}{Theorem}
\newtheorem{lemma}[theorem]{Lemma}
\newtheorem{remark}{Remark}
\newtheorem{example}{Example}
\newtheorem{assumption}{Assumption}
\newtheorem{definition}{Definition}

\begin{document}
\title[Pressure robust mixed methods for nearly incompressible elasticity]
{Pressure robust mixed methods for nearly incompressible elasticity}
\author{Seshadri R. Basava}
\address{Seshadri R. Basava, Department of Mathematics, Technische Universit\"at Darmstadt, Dolivostr. 15, 64293 Darmstadt, Germany}
\email{basava@mathematik.tu-darmstadt.de}
\author{Winnifried Wollner}
\address{Winnifried Wollner, Fachbereich Mathematik, MIN~Fakult\"at, Universit\"at Hamburg, Bundesstr. 55, 20146 Hamburg, Germany}
\email{winnifried.wollner@uni-hamburg.de}
\thanks{The authors thank Luca Heltai for helpful discussions on the
  implementation of the pressure robust interpolation in deal.ii.
  This work was funded by the Deutsche Forschungsgemeinschaft (DFG, German Research
Foundation) -- Projektnummer 392587580 -- SPP 1748}
\date{\today}
\begin{abstract}
  Within the last years pressure robust methods for the discretization of incompressible
  fluids have been developed. These methods allow the use of standard
  finite elements for the solution of the problem while simultaneously
  removing a spurious pressure influence in the approximation error of
  the velocity of the fluid, or the displacement of an incompressible solid.
  To this end, reconstruction operators are
  utilized mapping discretely divergence free functions to divergence
  free functions. This work shows that the modifications proposed for Stokes
  equation by \cite{Linke:2014} also yield gradient robust methods for nearly 
  incompressible elastic materials without the need to resort to discontinuous 
  finite elements methods as proposed in \cite{FuGuoshengLehrenfeldChristophLinkeStrechenbach:2020}.
\end{abstract}
\maketitle

\section{Introduction}
The Stokes equation for steady flow of an incompressible fluid is given as
\begin{equation}
  \begin{aligned}
    -\nu \Delta \mathbf{u} - \nabla p  & = \mathbf{f} & &\text{in }\Omega, \\
    \nabla \cdot \mathbf{u} & = 0 & &\text{in }\Omega, \\
    u &= 0 && \text{on }\Omega,
  \end{aligned}
\end{equation}
in a, polygonal, domain $\Omega \subset \mathbb{R}^d; d = 2,3$ for given data
$\mathbf{f} \in L^2(\Omega)$ and $\nu > 0$, where $\mathbf{u}$ denotes the fluid velocity and $p$ denotes 
the pressure. Under the famous inf-sup condition for the finite element spaces $\mathbf{V}_h$ and
$Q_h$, the use of mixed finite elements allows to obtain discrete
approximations $\mathbf{u}_h \in \mathbf{V}_h$ and $p_h \in Q_h$
satisfying an an error estimate of the form
\[
\|\mathbf{u}-\mathbf{u}_h\|_1 \le \frac{c}{\beta} \inf_{\mathbf{v}_h \in
  \mathbf{V}_h} \|\mathbf{u}-\mathbf{v}_h\|_1 + \frac{c}{\nu}\inf_{q_h \in Q_h}\|p-q_h\|_0,
\]
see, e.g.,~\cite{GiraultRaviart:1986}
Here $\beta$ is the inf-sup constant associated to the choice of
$\mathbf{V}_h$ and $Q_h$, $\|\,\cdot\,\|_1$ and $\|\,\cdot\,\|_0$ denote the
$H^1$ and $L^2$ norm on $\Omega$, respectively.
Further, here and throughout the paper $c$ denotes a generic constant
which is independent of all relevant quantities of the estimate but
may take a different value at each appearance.

While the estimate yields asymptotically
optimal orders without the need to utilize exactly divergence free
finite element functions for the approximation of $\mathbf{u}_h$ 
the right hand side of the estimate hints towards an undesirable
influence of the pressure on the approximation error of the
velocity. In fact, it has been observed, e.g., in~\cite{Linke:2014}
that indeed complicated pressures can give rise to a large error in
the velocity approximation, even in situations where the true
velocity can be represented in the discrete space $\mathbf{V}_h$.

A potential remedy, allowing for arbitrary inf-sup stable element pairs while providing
pressure independent velocity has been proposed by~\cite{Linke:2014}. He proposed the use of 
reconstruction operators on the right hand side of the equation to map discretely divergence free
functions to divergence free functions. This proposed method has been implemented to a range of 
problems and a variety of finite element pairs for the discretization of Stokes equation, such as
non-conforming Crouzeix-Raviart element~\cite{LinkeMerdonWollner:2017} , Taylor-Hood and MINI
elements with continuous pressure spaces~\cite{LedererLinkeMerdonSchoeberl:2017}, on rectangular
elements~\cite{LinkeMatthiesTobiska:2016}, for embedded discontinuous
Galerkin methods (EDG)~\cite{LedererRhebergen:2020}. For $3$-d polyhedral domains with concave edges a pressure robust 
reconstruction is given in~\cite{ApelKempf:2021}.
While the obtained convergence orders are optimal,
the price to pay, for these methods is a loss of quasi optimality of the method due to Strang's 
first lemma. Recently,~\cite{KreuzerVerfuerthZanotti:2020} showed that a more
involved construction of the reconstruction operator allows for a quasi-optimal discretization.

In this paper, we consider the extension of these results to nearly incompressible linear elasticity, e.g.,
\[
\begin{aligned}
  -2\mu \nabla \cdot \varepsilon(\mathbf{u}) - \lambda \nabla
  (\nabla \cdot \mathbf{u}) & = \mathbf{f} && \text{in }\Omega,\\
  u & = 0 && \text{on }\partial\Omega,
\end{aligned}
\]
where $\varepsilon(\mathbf{u})$ denotes the symmetric gradient, and
$\mu, \lambda > 0$ are the Lam\'e parameters.
To avoid the locking phenomenon, e.g.,~\cite[Chapter~VI.3]{Braess:2007},
typically a mixed form 
\begin{equation}\label{eq:elasticity_mixed}
  \begin{aligned}
    -2\mu \, \nabla \cdot \varepsilon(\mathbf{u}) - \nabla p & = \mathbf{f} &&  \text{ in } \Omega, \\
    \nabla \cdot \mathbf{u} - \frac{1}{\lambda} p & = 0 &&  \text{ in } \Omega, \\
    \mathbf{u} & = 0 &&  \text{ on } \partial \Omega,
  \end{aligned}
\end{equation}
is considered. Here the incompressible case, i.e., $\lambda=\infty$,
can easily be included by dropping the term $-\frac{1}{\lambda}p$ in
the second line. It is clear conceptually that the same difficulties
as for the Stokes problem will occur in the incompressible limit.
However, the treatment of the nearly incompressible case
requires additional care. To this
end,~\cite{FuGuoshengLehrenfeldChristophLinkeStrechenbach:2020}
defined a discretization to be ``gradient robust'', if the influence of
gradient forces $\mathbf{f} = \nabla \phi$ in the
discrete solution vanishes sufficiently fast as $\lambda \rightarrow
\infty$. \cite{FuGuoshengLehrenfeldChristophLinkeStrechenbach:2020}~showed
that a standard mixed discretization of~\eqref{eq:elasticity_mixed}
is not gradient robust and provided a gradient robust hybrid discontinuous Galerkin (HDG) scheme.
Within this article, we will show that mixed methods can be made
gradient robust using the approach proposed by~\cite{Linke:2014} for
the mixed discretization of~\eqref{eq:elasticity_mixed}.

The rest of the paper is structured as follows. In
Section~\ref{sec:discretization}, we introduce the notion of gradient
robustness and discuss the discretization
of~\eqref{eq:elasticity_mixed}. Next, in
Section~\ref{sec:erroranalysis}, we show that the proposed
discretization is indeed gradient robust and provide error estimates.
We conclude the paper with a series of examples highlighting the
derived results in Section~\ref{sec:numerics}.

\section{Gradient Robustness and Discretization}\label{sec:discretization}

\subsection{Gradient Robustness}
We define the spaces $\mathbf{V}^0$ of divergence free function and
its orthogonal complement $\mathbf{V}^\bot$ as
\begin{align*}
  \mathbf{V}^{0} &= \left\{ \mathbf{u} \in H^1_0(\Omega;\mathbb{R}^d) : \nabla \cdot \mathbf{u} = 0\right\}, \label{eq:vzero}\\
  \mathbf{V}^\bot &= \left\{ \mathbf{u} \in H^1_0(\Omega;\mathbb{R}^d) : a(\mathbf{u}, \mathbf{v}) = 0 , \forall \, \mathbf{v} \in \mathbf{V}^{0} \right\},
\end{align*}
where for $\mathbf{u}, \mathbf{v}\in \mathbf{V} = H^1_0(\Omega;\mathbb{R}^d)$, we define the bilinear form (scalar product)
$a \colon \mathbf{V} \times \mathbf{V} \rightarrow \mathbb{R}$ by 
\begin{equation}\label{eq:bilinear_a}
  a(\mathbf{u}, \mathbf{v}) = 2\mu (\varepsilon(\mathbf{u}),
  \varepsilon(\mathbf{v})),
\end{equation}
with the $L^2(\Omega)$-scalar product $(\,\cdot\,,\,\cdot\,)$.
Now, any function $\mathbf{u} \in \mathbf{V}$ can be uniquely written as 
$\mathbf{u} = \mathbf{u}^0 + \mathbf{u}^\bot \in \mathbf{V}^0 \oplus \mathbf{V}^\bot$.

Using Helmholtz decomposition, $\mathbf{f} \in L^2(\Omega)$ can be uniquely decomposed as
\begin{equation}
  \mathbf{f} = \nabla \phi + \mathbf{w},
\end{equation}
where $\nabla \phi \in H^1(\Omega)/ \mathbb{R}$ is irrotational, $\mathbf{w}$ is divergence free and
both are orthogonal with respect to the $L^2(\Omega)$-scalar product, i.e., 
\begin{equation}
  (\mathbf{w},\nabla \, \phi) = 0. 
\end{equation}
With these definitions, the decay of the influence of gradient forces,
i.e., $\mathbf{w} = 0$, onto the solutions $\mathbf{u}$ of~\eqref{eq:elasticity_mixed} can be
quantified as the following result
from~\cite[Theorem~1]{FuGuoshengLehrenfeldChristophLinkeStrechenbach:2020}
shows:
\begin{lemma}\label{lem:u0}
  If $\mathbf{f} \in H^{-1}(\Omega)$ is a gradient, i.e., $\mathbf{f} = \nabla \phi,$ for some $ \phi \in L^2(\Omega)$. Then 
  for the solution $\mathbf{u} = \mathbf{u}^0 + \mathbf{u}^\bot$
  of~\eqref{eq:elasticity_mixed} it holds $\mathbf{u}^0 = 0$ and
  \[
  \|\mathbf{u}\|_1 = \|\mathbf{u}^\bot\|_1 \le \frac{c}{\mu+
    \lambda}\|\phi\|_0.
  \]
  In particular, $\|\mathbf{u}\|_1 = O(\lambda^{-1})$ as $\lambda
  \rightarrow \infty$.
\end{lemma}
Since this bound need not hold for arbitrary
discretizations,~\cite{FuGuoshengLehrenfeldChristophLinkeStrechenbach:2020}
introduced the following notion
\begin{definition}\label{def:gradientrobust}
  A discretization of~\eqref{eq:elasticity_mixed} is called
  \emph{gradient robust}, if for any discretization parameter $h$
  there is a constant $c_h$ such that
  the approximate solution $\mathbf{u}_h$ satisfies
  \[
  \|\mathbf{u}_h\|_1 \le \frac{c_h}{\lambda}\|\phi\|_0
  \]
\end{definition}
\subsection{Abstract Discretization}
In order to discretize~\eqref{eq:elasticity_mixed}, we define a second
bilinear form $b\colon Q \times \mathbf{V} \rightarrow \mathbb{R}$,
with $Q=L^2(\Omega)$, by
\begin{equation}\label{eq:bilinear_b}
  b(q,\mathbf{v}) = (p,\nabla \cdot \mathbf{v}).
\end{equation}
Now we select subspaces $\mathbf{V}_h \subset \mathbf{V}$ and $Q_h
\subset Q$ such that there is a positive constant $\beta$ satisfying
the inf-sup condition
\begin{equation}\label{eq:inf-sup}
  \inf_{q_h \in Q_h} \sup_{\mathbf{v}_h \in \mathbf{V}_h} \frac{\left(q_h , \nabla \cdot \mathbf{v}_h\right)}{\|q_h\|_0 \lvert\mathbf{v}_h\rvert_1} \geq \beta.
\end{equation} 

Now, the standard, non gradient robust, weak formulation is given as
follows:
Find $(\mathbf{u}_h, p_h) \in \mathbf{V}_h \times Q_h$ such that
\begin{equation} \label{eq:weak_form}
  \begin{aligned}
    a(\mathbf{u}_h,\mathbf{v}_h) + b(p_h, \mathbf{v}_h)  & = (\mathbf{f},\mathbf{v}_h) &
    \forall& \mathbf{v}_h \in \mathbf{V}_h, \\
    b(q_h,\mathbf{u}_h) - \frac{1}{\lambda} (p_h,q_h)& = 0 &
    \forall& q_h \in Q_h.
  \end{aligned}
\end{equation}
Under the well known inf-sup condition~\eqref{eq:inf-sup} on
$\mathbf{V}_h$ and $Q_h$, the system~\eqref{eq:weak_form} is uniquely
solvable~\cite[Theorem~5.5.2]{BoffiBrezziFortin:2013}. Following~\cite[Proposition~5.5.3]{BoffiBrezziFortin:2013}
the displacement error is thus bounded as follows:
\begin{equation} \label{eq:disp_error}
  \|\mathbf{u} - \mathbf{u}_h\|_1 \le \frac{c}{\beta}\inf_{\mathbf{v}_h \in \mathbf{V}_h} \|\mathbf{u} - \mathbf{v}_h\|_1 + 
  \frac{1}{\mu}\left(\frac{1}{\lambda} + 1\right)\inf_{q_h \in Q_h}\|p - q_h\|_0.
\end{equation}
Following~\cite{Linke:2014}, we assume that there exists a reconstruction operator
\begin{equation*}
  \mathbf{\pi}^{\rm{div}} \colon \mathbf{V}_h \rightarrow
  H^{\rm{div}}(\Omega)
  = \left\{ \mathbf{v} \in L^2(\Omega)^d\,:\, \nabla \cdot \mathbf{v} \in L^2(\Omega)\right\},
\end{equation*}
to be specified later in Section~\ref{sec:reconstruction},
mapping discretely divergence free functions to divergence free
functions. Then the modified problem is given as:
\begin{equation} \label{eq:modified_weak_form}
  \begin{aligned}
    a(\mathbf{u}_h,\mathbf{v}_h) + b(p_h, \mathbf{v}_h)  & = (\mathbf{f},\mathbf{\pi}^{\rm{div}}\mathbf{v}_h) &
    \forall& \mathbf{v}_h \in \mathbf{V}_h, \\
    b(q_h,\mathbf{u}_h) - \frac{1}{\lambda} (p_h,q_h)& = 0 &
    \forall& q_h \in Q_h.
  \end{aligned}
\end{equation}
Clearly, by construction, the modified
problem~\eqref{eq:modified_weak_form} admits a solution under the
same conditions as~\eqref{eq:weak_form}, since only the right hand
side has been modified.
In Theorem~\ref{thm:gradient_robust}, we will see that the
discretization~\eqref{eq:modified_weak_form} is gradient robust,
under appropriate assumptions on $\mathbf{\pi}^{\rm{div}}$.
Further, in Theorem~\ref{thm:gradient_robust_error}, we show the gradient robust displacement error estimate 
\begin{equation*}
  \|\mathbf{u} - \mathbf{u}_h\|_1 \le ch^k \left(1 + \sqrt{\frac{\mu}{\lambda}}\right)\|\mathbf{u}\|_{k+1} +  
  c\frac{h^k}{\mu\lambda}\|p\|_k,
\end{equation*}
where $\|\,\cdot\,\|_k$ denotes the norm on $H^{k}(\Omega)$ or
$H^{k}(\Omega;\mathbb{R}^d)$; of course assuming sufficient
regularity of $\mathbf{u}$ and $p$ and approximation order of
$\mathbf{V}_h$ and $Q_h$.
\subsection{Reconstruction Operator and Assumptions}\label{sec:reconstruction}
The construction of the reconstruction operator
$\mathbf{\pi}^{\rm{div}}$ proposed by~\cite{Linke:2014} is based on
the choice of a suitable subspace $\mathcal{M}_h \subset
H^{\rm{div}}(\Omega)$ satisfying the commuting diagram in
Figure~\ref{fig:comm_diag}
where $\pi^{L^2}$ denotes the $L^2$-projection onto $Q_h$. 
\begin{figure}[H]
  \centering
  \includegraphics[width=0.3\textwidth]{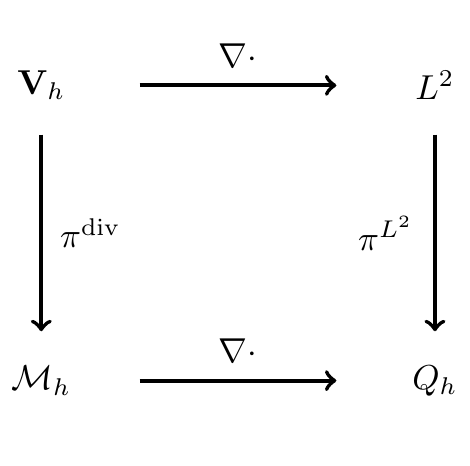} 
  \caption{Commutative diagram for the reconstruction operator $\mathbf{\pi}^{\rm{div}}$}
  \label{fig:comm_diag}
\end{figure}
The commuting diagram is equivalently expressed by the equation
\begin{equation}\label{eq:comm_prop}
  b(q_h,\mathbf{\pi}^{\rm{div}}\mathbf{v}_h) =
  b(q_h,\mathbf{v}_h) \qquad \forall \mathbf{v}_h \in \mathbf{V}_h, q_h
  \in Q_h,
\end{equation}
holds.
Moreover, defining
\begin{align}
  \mathbf{V}_h^0 &= \{ \mathbf{v}_h \in \mathbf{V}_h\,:\,
  b(q_h,\mathbf{v}_h) = 0 \; \forall q_h \in Q_h\} \label{eq:Vh0},\\
  H^{\rm{div}}_0(\Omega) &= \{\mathbf{v}\in H^{\rm{div}}(\Omega)\,:\, \nabla \cdot v = 0\},
\end{align}
we require that the restriction of $\mathbf{\pi}^{\rm{div}}$
to discretely divergence free functions maps into divergence free
functions, i.e.,
\begin{equation}\label{eq:compatibility_divergence}
  \mathbf{\pi}^{\rm{div}}\colon \mathbf{V}_h^0 \rightarrow H^{\rm{div}}_0(\Omega)
\end{equation}
and further for any $\mathbf{v}_h \in \mathbf{V}_h$ it holds 
\begin{equation}\label{eq:compatibility_boundarytrace}
  \mathbf{\pi}^{\rm{div}}\mathbf{v}_h \cdot \mathbf{n} = 0 \quad
  \text{on }\partial \Omega.
\end{equation}
Analogously to the continuous setting, we can define the orthogonal
complement $\mathbf{V}_h^\bot$ by
\[
\mathbf{V}_h^\bot = \left\{ \mathbf{u}_h \in \mathbf{V}_h :
a(\mathbf{u}_h, \mathbf{v}_h) = 0 , \forall \, \mathbf{v}_h \in
\mathbf{V}_h^{0} \right\},
\]
and the corresponding discrete decomposition $\mathbf{u}_h = \mathbf{u}_h^0
+ \mathbf{u}_h^\bot \in \mathbf{V}_h^0 \oplus \mathbf{V}_h^\bot$.

Before we continue, let us make some, generic assumptions on the
considered spaces $\mathbf{V}_h$ and $Q_h$ defined on a shape regular
family $\mathcal T_h$ of decompositions of $\Omega$.
\begin{assumption}\label{ass:fespace}
  We assume, that for some $k \ge 2$ the finite element space
  $\mathbf{V}_h$ is equipped with an interpolation operator
  $I_h \colon H^{k+1}(\Omega;\mathbb{R}^d) \rightarrow \mathbf{V}_h$
  satisfying 
  \[
  h_T^i\|I_h \mathbf{v} - \mathbf{v}\|_{i,T} \le
  ch_T^{k+1}\|\mathbf{v}\|_{k+1,T} \qquad\forall \mathbf{v}\in
  H^{k+1}(\Omega;\mathbb{R}^d), T \in \mathcal{T}_h
  \]
  where $\|\,\cdot\,\|_{i,T}$ denotes the respective norm on the
  element $T$, and $h_T$ is the element diameter.
  For the space $Q_h$, we assume that the $L^2$-projection
  $\pi^{L^2}=\pi^{L^2}_{k-1} \colon H^{k}(\Omega) \rightarrow Q_h$
  satisfies
  \[
  h_T^i\|\pi^{L^2} q - q\|_{i,T} \le
  ch_T^{k}\|q\|_{k,T} \qquad\forall q\in
  H^{k}(\Omega), T \in \mathcal{T}_h.
  \]
  Further, it is assumed that $\mathbf{V}_h$ and $Q_h$ satisfy the
  inf-sup inequality~\eqref{eq:inf-sup}.
  Finally, we assume that there exists a subspace $\widetilde{\mathbf{Q}}_h
  \subset L^2(\Omega; \mathbb{R}^d)$ such that the respective $L^2$-projection $\pi^{L^2}=\pi^{L^2}_{k-2}$ satisfies
  \[
  h_T^i\|\pi^{L^2} \mathbf{q} - \mathbf{q}\|_{i,T} \le
  ch_T^{k-1}\|\mathbf{q}\|_{k-1,T} \qquad\forall \mathbf{q}\in
  H^{k}(\Omega;\mathbb{R}^d), T \in \mathcal{T}_h.
  \]
  Further requirements on $\widetilde{\mathbf{Q}}_h$ will be made in Assumption~\ref{ass:1}.
\end{assumption}
With these preparations, we can now state the additional assumptions
on the recovery operator.
\begin{assumption}\label{ass:1}
  We first assume, that the recovery operator satisfies the following
  orthogonality relation
  \begin{equation}
    \left(\mathbf{v}_h - \mathbf{\pi}^{\rm{div}} \mathbf{v}_h,
    \mathbf{q}\right)  = 0\qquad  \forall \mathbf{v}_h \in \mathbf{V}_h,
    \mathbf{q}\in \widetilde{\mathbf{Q}}_h, \label{eq:ass_ortho}
  \end{equation}
  where $\widetilde{\mathbf{Q}}_h \subset L^2(\Omega;\mathbb{R}^d)$ is given in Assumption~\ref{ass:fespace}.
  Second, we assume the following local approximation property to hold
  \begin{equation}
    \|\mathbf{\pi}^{\rm{div}} \mathbf{v}_h - \mathbf{v}_h\|_{0,T}
    \le ch^m_T\lvert\mathbf{v}_h\rvert_{m,T} \qquad \forall \;
    \mathbf{v}_h \in \mathbf{V}_h, T \in \mathcal{T}_h, m = 0,
    1.  \label{eq:ass_approx}
  \end{equation}
\end{assumption}
Before concluding the assumption, let us note that the assumptions
can indeed be satisfied. To this end, we give an example which we
will also use for the numerical results in Section~\ref{sec:numerics}.
\begin{example}
  Let us assume that the domain can be decomposed into a
  family $\mathcal T_h$ of shape regular rectangular ($d=2$)
  or brick ($d=3$) elements. 
  For the space $\mathbf{V}_h = \mathbf{V}_h^k$, we consider, parametric, piecewise
  $\mathcal Q_k$ and globally continuous finite elements with $k \ge 2$.
  For the discretization of $Q_h = Q_h^{k-1}$, we select
  the space of discontinuous piecewise $P_{k-1}$ functions.
  Indeed theses pairs satisfy the inf-sup condition~\eqref{eq:inf-sup},
  see, e.g.,~\cite[Sec.~8.6.3 \& 8.7.2]{BoffiBrezziFortin:2013} for
  $k=2$, for arbitrary $k$~\cite[Sec.~3.2]{GiraultRaviart:1986}
  or~\cite{MatthiesTobiska:2002} for mapped pressure spaces.
  Moreover,~\cite[Sec.~4.2.1]{LinkeMatthiesTobiska:2016} showed, that
  the choice $\mathcal{M}_h = \mathcal{BDM}_k$ as space of
  Brezzi-Douglas-Marini elements yield the desired commuting diagram
  property~\eqref{eq:comm_prop} together with the canonical
  interpolation $\mathbf{\pi}^{\rm{div}}$.
  Further, they showed~\cite[Lemma~2.1]{LinkeMatthiesTobiska:2016},
  that the restriction of $\mathbf{\pi}^{\rm{div}}$
  to discretely divergence free functions maps into divergence free
  functions, i.e.,
  \[
  \mathbf{\pi}^{\rm{div}}\colon \{ \mathbf{v}_h \in \mathbf{V}_h\,:\,
  b(q_h,\mathbf{v}_h)\; \forall q_h \in Q_h\} \rightarrow
  \{\mathbf{v}\in H^{\rm{div}}(\Omega)\,:\, \nabla \cdot v = 0\}
  \]
  and further for any $\mathbf{v}_h \in \mathbf{V}_h$ it holds 
  \[
  \mathbf{\pi}^{\rm{div}}\mathbf{v}_h \cdot \mathbf{n} = 0 \quad
  \text{on }\partial \Omega.
  \]
\end{example}

\begin{remark}
  Infact, \cite{LinkeMatthiesTobiska:2016} showed
  that~\eqref{eq:comm_prop} follow from a set of assumed
  orthogonality properties and surjectivity of divergence and normal
  traces from which suitable choices of $\mathcal{M}_h$ and
  constructions of $\mathbf{\pi}^{\rm{div}}$ can be obtained.
\end{remark}

\section{Error Analysis}\label{sec:erroranalysis}
In this section, we proceed with error analysis of the modified weak
form~\eqref{eq:modified_weak_form}.
We split the analysis in two parts for incompressible materials $(\lambda = \infty)$ and nearly incompressible materials
$(\lambda \ne \infty)$. 

\subsection{Incompressible Materials}\label{sec:incompress}
We proceed to the error analysis of incompressible materials, where
$\lambda = \infty$ and the term involving $\frac{1}{\lambda}$ is
dropped in~\eqref{eq:modified_weak_form}. The analysis follows, at
large, the arguments in~\cite{LinkeMerdonWollner:2017} with some
minor adjustments to the elasticity case.

\begin{theorem}\label{thm:thm1_infty}
  Let Assumptions~\ref{ass:fespace} and~\ref{ass:1} be satisfied and
  $\lambda = \infty$. Then 
  the solution $(\mathbf{u}, p) \in H^{k+1}(\Omega)^d \times
  H^k(\Omega)$ of the continuous
  problem~\eqref{eq:elasticity_mixed}
  and the solution $(\mathbf{u}_h, p_h) \in \mathbf{V}_h\times
  Q_h$ of~\eqref{eq:modified_weak_form}
  satisfy the error estimate
  \begin{equation*}
    \lvert\mathbf{u} - \mathbf{u}_h\rvert_{1}^2 \le c\sum\limits_{T
      \in \mathcal{T}_h} h^{2k}_T\lvert\mathbf{u}\rvert^2_{k+1,T}
    \le c h^{2k}\|\mathbf{u}\|_{k+1},
  \end{equation*}
  where $\lvert \,\cdot\, \rvert_1$ denotes the $H^1$-semi-norm.
\end{theorem}

Before proving the above theorem, we would like to prove an important lemma which is need to prove the theorem.
\begin{lemma}\label{lem:consistency}
  Let Assumptions~\ref{ass:fespace} and~\ref{ass:1} be satisfied and
  $\lambda = \infty$.
  Then for any functions $ \mathbf{u} \in H^{k+1}(\Omega)^d$ and
  $\mathbf{w}_h\in\mathbf{V}_h$
  it is
  \begin{equation}
    \bigl\lvert(\nabla  \cdot  \varepsilon(\mathbf{u}), \mathbf{\pi}^{\rm{div}} \mathbf{w}_h) + (\varepsilon(\mathbf{u}), \varepsilon(\mathbf{w}_h))\bigr\rvert \le
    c \sum\limits_{T \in \mathcal{T}_h} h_T^k\lvert\mathbf{u}\rvert_{k+1,T} \lvert\mathbf{w}_h\rvert_{1,T} \label{eq:lemma2},
  \end{equation}
  where $\lvert \,\cdot\, \rvert_{i,T}$ denotes the $H^i$-semi-norm on $T$.
\end{lemma}

\begin{proof}
  We add and subtract $\left(\nabla  \cdot  \varepsilon(\mathbf{u}) ,
  \mathbf{w}_h\right)$ on the left to obtain
  \begin{equation}\label{eq:lemma2_proof1}
    \begin{aligned}
      (\nabla  \cdot  \varepsilon(\mathbf{u}),
      \mathbf{\pi}^{\rm{div}} \mathbf{w}_h) +
      (\varepsilon(\mathbf{u}), \varepsilon(\mathbf{w}_h)) 
      =&\; (\nabla  \cdot \varepsilon(\mathbf{u}), \mathbf{\pi}^{\rm{div}}\mathbf{w}_h - \mathbf{w}_h)\\
      &\;+(\varepsilon(\mathbf{u}), \varepsilon(\mathbf{w}_h)) + (\nabla  \cdot  \varepsilon(\mathbf{u}), \mathbf{w}_h).
    \end{aligned}
  \end{equation}
  Since $\nabla  \cdot  \varepsilon(\mathbf{u}) \in L^2(\Omega;\mathbb{R}^d)$, we
  can apply the projection $\pi^{L^2}_{k-2}$, from
  Assumption~\ref{ass:fespace}, to get
  $\mathbf{\pi}^{L^2}_{k-2} \nabla  \cdot  \varepsilon(\mathbf{u}) \in  \widetilde{\mathbf{Q}}_h$.
  By the assumed orthogonality in~\eqref{eq:ass_ortho}, we have
  \[\left(\mathbf{\pi}^{L^2}_{k-2} \nabla  \cdot
  \varepsilon(\mathbf{u}), \mathbf{\pi}^{\rm{div}}\mathbf{w}_h -
  \mathbf{w}_h\right) = 0, \qquad \forall \; \mathbf{w}_h \in
  \mathbf{V}_h .
  \]
  Using Assumption~\ref{ass:fespace} and~\eqref{eq:ass_approx}, we
  obtain, for the first summand on the right
  of~\eqref{eq:lemma2_proof1}, 
  \begin{equation}\label{eq:lemma2_proof2}
    \begin{aligned}
      \Bigl( \nabla  \cdot  \varepsilon(\mathbf{u})&,
      \mathbf{\pi}^{\rm{div}} \mathbf{w}_h - \mathbf{w}_h\Bigr)
      =  \left(\nabla  \cdot  \varepsilon(\mathbf{u}) - \mathbf{\pi}^{L^2}_{k-2} \nabla  \cdot  \varepsilon(\mathbf{u}),
      \mathbf{\pi}^{\rm{div}} \mathbf{w}_h - \mathbf{w}_h\right)  \\
      & \le \sum\limits_{T \in \mathcal{T}_h} \|\nabla  \cdot  \varepsilon(\mathbf{u}) - \mathbf{\pi}^{L^2}_{k-2} \nabla  \cdot  \varepsilon(\mathbf{u})\|_{0, T}
      \|\mathbf{\pi}^{\rm{div}}\mathbf{w}_h - \mathbf{w}_h\|_{0, T} \\
      & \le \sum\limits_{T\in\mathcal{T}_h} c h_T^{k-1}\lvert\nabla  \cdot  \varepsilon(\mathbf{u})\rvert_{k-1,T} 
      h\lvert\mathbf{w}_h\rvert_{1,T} \\
      & \le \sum\limits_{T\in\mathcal{T}_h} ch_T^k\lvert\mathbf{u}\rvert_{k+1,T}\lvert\mathbf{w}_h\rvert_{1,T}. 
    \end{aligned}
  \end{equation}
  For the last two summands of~\eqref{eq:lemma2_proof1}, we apply
  Gauss divergence theorem to get
  \begin{equation}\label{eq:lemma2_proof3}
    (\nabla  \cdot  \varepsilon(\mathbf{u}), \mathbf{w}_h) + (\varepsilon(\mathbf{u}), \varepsilon(\mathbf{w}_h)) 
    = \int\limits_{\partial \Omega} \varepsilon(\mathbf{u})\cdot
    \mathbf{n} \; \mathbf{w}_h \,\mathrm{d}s = 0 
  \end{equation}
  since $\mathbf{w}_h = 0$ on $\partial \Omega$.
  Combining~\eqref{eq:lemma2_proof1} with the
  bounds~\eqref{eq:lemma2_proof2} and~\eqref{eq:lemma2_proof3} the
  assertion is shown.
\end{proof}

Now, we continue to prove Theorem~\ref{thm:thm1_infty}
\begin{proof} (of Theorem~\ref{thm:thm1_infty})
  Let $\mathbf{u}_h$ be the solution
  of~\eqref{eq:modified_weak_form}, with $\lambda = \infty$,
  and let $\mathbf{v}_h \in \mathbf{V}_h^0$ be arbitrary.
  Defining $\mathbf{w}_h = \mathbf{u}_h - \mathbf{v}_h \in
  \mathbf{V}_h^0$ and applying the triangle inequality gives 
  \begin{equation}
    \lvert\mathbf{u} - \mathbf{u}_h\rvert_1 = \lvert\mathbf{u} - \mathbf{w}_h -\mathbf{v}_h\rvert_{1} \le \lvert\mathbf{u} - \mathbf{v}_h\rvert_1 + \lvert\mathbf{w}_h\rvert_1. \label{eq:thm1_proof1}
  \end{equation}
  In view of the interpolation estimate in
  Assumption~\ref{ass:fespace}, we are left to 
  estimate $\lvert\mathbf{w}_h\rvert_1$. From Korn's inequality, we have
  \[
  c\lvert\mathbf{w}_h\rvert^2_1 = c \|\mathbf{w}_h\|^2_1 \le
  \|\varepsilon(\mathbf{w}_h)\|^2_0 .
  \]
  From this, we conclude
  \begin{equation}\label{eq:thm1_proof2}
    \begin{aligned}
      2\mu c\lvert\mathbf{w}_h\rvert^2_1 & \le a(\mathbf{w}_h, \mathbf{w}_h) \\
      & = a(\mathbf{u}_h -\mathbf{v}_h, \mathbf{w}_h) \\
      & = a(\mathbf{u}_h - \mathbf{v}_h, \mathbf{w}_h) \\
      & = a(\mathbf{u}_h -\mathbf{v}_h + \mathbf{u}- \mathbf{u}, \mathbf{w}_h) \\
      & \le \lvert a(\mathbf{u} -\mathbf{v}_h, \mathbf{w}_h)\rvert  + 
      \lvert a(\mathbf{u}_h - \mathbf{u}, \mathbf{w}_h)\rvert. 
    \end{aligned}
  \end{equation}
  For the first summand on the right of~\eqref{eq:thm1_proof2} we use
  Cauchy-Schwartz inequality to get
  \begin{equation}\label{eq:thm1_proof3}
    \lvert a(\mathbf{u} - \mathbf{v}_h,
    \mathbf{w}_h)\rvert
    \le
    2 \mu \|\varepsilon(\mathbf{u} - \mathbf{v}_h)\|_0\|\varepsilon(\mathbf{w}_h)\|_0, 
    \le 2 \mu \lvert\mathbf{u} - \mathbf{v}_h\rvert_1 \lvert\mathbf{w}_h\rvert_1. 
  \end{equation}
  
  Before we come to the bound of the second summand
  in~\eqref{eq:thm1_proof2}, we make some preliminary calculations.
  Since $\mathbf{u}_h$ is the solution
  of~\eqref{eq:modified_weak_form}, choosing $\mathbf{v}_h =
  \mathbf{w}_h \in \mathbf{V}_h^0$ gives 
  \begin{equation}\label{eq:thm1_proof4}
    a(\mathbf{u}_h,\mathbf{w}_h)
    = a(\mathbf{u}_h,\mathbf{w}_h) + b(p_h,\mathbf{w}_h)
    = (\mathbf{f}, \mathbf{\pi}^{\rm{div}} \mathbf{w}_h).
  \end{equation}
  Further, since $\mathbf{u}$ is the solution to the equation~\eqref{eq:elasticity_mixed}
  multiplication with $\mathbf{\pi}^{\rm{div}} \mathbf{w}_h$ and
  integration yields
  \begin{equation*}
    -2\mu \int\limits_\Omega \nabla  \cdot  \varepsilon(\mathbf{u}) \mathbf{\pi}^{\rm{div}}\mathbf{w}_h \,\mathrm{d}x - 
    \int\limits_\Omega \nabla p \;\mathbf{\pi}^{\rm{div}} \mathbf{w}_h \,\mathrm{d}x 
    = \int\limits_\Omega  \mathbf{f} \mathbf{\pi}^{\rm{div}} \mathbf{w}_h \,\mathrm{d}x  
  \end{equation*}
  by the compatibility of the reconstruction with the kernel of the
  divergence, i.e.,~\eqref{eq:compatibility_divergence}, this gives
  \begin{equation*}
    -2\mu ( \nabla  \cdot  \varepsilon(\mathbf{u}), \mathbf{\pi}^{\rm{div}}\mathbf{w}_h) = (\mathbf{f}, \mathbf{\pi}^{\rm{div}}\mathbf{w}_h) 
  \end{equation*}
  Combining this with~\eqref{eq:thm1_proof4}, we get
  \begin{equation}\label{eq:thm1_proof5}
    a(\mathbf{u}_h, \mathbf{w}_h)  = -2\mu(\nabla  \cdot  \varepsilon(\mathbf{u}), \mathbf{\pi}^{\rm{div}} \mathbf{w}_h). 
  \end{equation}

  Now, we can bound the second summand on the right
  of~\eqref{eq:thm1_proof2}, using~\eqref{eq:thm1_proof5} we get
  \begin{equation*}
    \begin{aligned}
      \lvert a(\mathbf{u}_h - \mathbf{u}, \mathbf{w}_h)\rvert  & 
      = \Bigl\lvert-2\mu (\nabla  \cdot  \varepsilon(\mathbf{u}), \mathbf{\pi}^{\rm{div}} \mathbf{w}_h) - 
      2 \mu (\varepsilon(\mathbf{u}), \varepsilon(\mathbf{w}_h))\Bigr\rvert \\
      & \le 2 \mu \Bigl\lvert(\nabla  \cdot  \varepsilon(\mathbf{u}), \mathbf{\pi}^{\rm{div}} \mathbf{w}_h) + 
      (\varepsilon(\mathbf{u}), \varepsilon(\mathbf{w}_h))\Bigr\rvert. 
    \end{aligned}
  \end{equation*}
  By the previously shown lemma, i.e.,~\eqref{eq:lemma2}, we can bound
  the right hand side to get 
  \begin{equation}\label{eq:thm1_proof6}
    \begin{aligned}
      \lvert a(\mathbf{u}_h - \mathbf{u}, \mathbf{w}_h)\rvert  & \le 
      2\mu c \sum\limits_{T \in \mathcal{T}_h} \left( h^k_T\lvert\mathbf{u}\rvert_{k+1,T}\lvert\mathbf{w}_h\rvert_{1,T} \right) \\
      & \le 2\mu c \left( \sum\limits_{T \in \mathcal{T}_h}  h^{2k}_T\lvert\mathbf{u}\rvert^2_{k+1,T}\right)^{\frac{1}{2}}\lvert\mathbf{w}_h\rvert_1.
    \end{aligned}
  \end{equation}
  
  Now combining~\eqref{eq:thm1_proof2} with the two
  bounds~\eqref{eq:thm1_proof3} and~\eqref{eq:thm1_proof6}, we get 
  \begin{equation*}
    \lvert\mathbf{w}_h\rvert_1  \le \lvert\mathbf{u} - \mathbf{v}_h\rvert_1 + c\left( \sum\limits_{T \in \mathcal{T}_h} h^{2k}_T \lvert\mathbf{u}\rvert^2_{k+1, T}\right)^{\frac{1}{2}}.  
  \end{equation*}
  Substituting this in~\eqref{eq:thm1_proof1} yields  
  \begin{equation}\label{eq:thm1_proof7}
    \lvert\mathbf{u}- \mathbf{u}_h\rvert_1  \le 2\lvert\mathbf{u} - \mathbf{v}_h\rvert_1 + c \left( \sum\limits_{T \in \mathcal{T}_h} h^{2k}_T \lvert\mathbf{u}\rvert^2_{k+1, T}\right)^{\frac{1}{2}}.
  \end{equation}
  To bound the best approximation error on $\mathbf{V}_h^0$ in this inequality,
  we proceed using inf-sup condition
  as in~\cite[Chapter~2, (1.16)]{GiraultRaviart:1986} and the assumed
  interpolation estimate on $\mathbf{V}_h$ in Assumption~\ref{ass:fespace}, to get the
  estimate
  \begin{equation*}
    \inf\limits_{\mathbf{v}_h \in \mathbf{V}_h^0} \lvert\mathbf{u} - \mathbf{v}_h\rvert_1 \le c \inf\limits_{\mathbf{v}_h \in \mathbf{V}_h} \lvert\mathbf{u} - \mathbf{v}_h\rvert_1 
    \le c\left( \sum\limits_{T \in \mathcal{T}_h}  h^{2k}_T \lvert\mathbf{u}\rvert^2_{k+1, T}\right)^{\frac{1}{2}}.
  \end{equation*}
  Using this in~\eqref{eq:thm1_proof7} gives the desired estimate.
\end{proof}

\subsection{Nearly Incompressible Materials}
For the nearly incompressible case, i.e., $(\lambda \neq \infty)$, we
start by assuming a gradient force $\mathbf{f} = \nabla \phi$, for
some $\phi \in L^2(\Omega)$.  
From Lemma~\ref{lem:u0}, we have that the solution
of~\eqref{eq:elasticity_mixed} for such an $\mathbf{f}$ is
$\mathbf{u} = 0$. The following result shows, that out mixed
discretization~\eqref{eq:modified_weak_form} is gradient robust
in the sense of Definition~\ref{def:gradientrobust}.

\begin{theorem}\label{thm:gradient_robust}
  Let Assumptions~\ref{ass:fespace} and~\ref{ass:1} be satisfied.
  If the right hand side $\mathbf{f} \in
  H^{-1}(\Omega;\mathbb{R}^d)$
  of equation~\eqref{eq:modified_weak_form}
  is a gradient field, i.e.,
  $\mathbf{f} = \nabla \phi$, for some $\phi \in L^2(\Omega)$,
  then the solution $(\mathbf{u}_h, p_h) \in \mathbf{V}_h \times
  Q_h$ of~\eqref{eq:modified_weak_form} with $\lambda \neq \infty$ satisfies
  \begin{equation}\label{eq:thm4_error}
    \|\mathbf{u}_h\|_1\le  c\frac{1}{\lambda + \mu} \|\phi\|_0.
  \end{equation}
  with a constant $c$ independent of $h$.
\end{theorem}
\begin{proof} 
  Consider $\mathbf{v}_h = \mathbf{u}_h$ in
  equation~\eqref{eq:modified_weak_form} with $\mathbf{f} = \nabla
  \phi$. Then integration by parts for the right hand side, using
  the zero trace from~\eqref{eq:compatibility_boundarytrace}, we get
  \begin{equation}\label{eq:thm_gradient_robust_eq1}
    a(\mathbf{u}_h, \mathbf{u}_h) + b(p_h, \mathbf{u}_h) 
    = - ( \phi, \nabla \cdot \mathbf{\pi}^{\rm{div}} \mathbf{u}_h ). 
  \end{equation}
  Since $\nabla \cdot \mathbf{\pi}^{\rm{div}} \mathbf{u}_h \in Q_h$
  we can rewrite the right hand side as
  \begin{equation}\label{eq:thm_gradient_robust_eq2}
    (\phi, \nabla \cdot \mathbf{\pi}^{\rm{div}} \mathbf{u}_h)  = (\pi^{L^2} \phi, \nabla \cdot \mathbf{\pi}^{\rm{div}} \mathbf{u}_h). 
  \end{equation}
  Since $\pi^{L^2} \nabla \cdot \mathbf{u}_h \in Q_h$, we can
  use it to test the second line in~\eqref{eq:modified_weak_form} giving
  \begin{equation}\label{eq:nearly_incompress_q}
    \begin{aligned}
      (\pi^{L^2} \nabla \cdot \mathbf{u}_h , \pi^{L^2} \nabla \cdot \mathbf{u}_h) & 
      = (\pi^{L^2} \nabla \cdot \mathbf{u}_h , \nabla \cdot \mathbf{u}_h ) \\ 
      &=\frac{1}{\lambda} (p_h, \pi^{L^2}\nabla \cdot
      \mathbf{u}_h)\\
      &=\frac{1}{\lambda} (p_h, \nabla \cdot
      \mathbf{u}_h)\\
      &=\frac{1}{\lambda} b(p_h,\mathbf{u}_h).
    \end{aligned}
  \end{equation}
  Substituting~\eqref{eq:thm_gradient_robust_eq2}  and~\eqref{eq:nearly_incompress_q} in~\eqref{eq:thm_gradient_robust_eq1}, we get
  \begin{equation}
    a(\mathbf{u}_h, \mathbf{u}_h) + \lambda(\pi^{L^2} \nabla \cdot \mathbf{u}_h , \pi^{L^2} \nabla \cdot \mathbf{u}_h ) = -(\pi^{L^2} \phi, \nabla \cdot \mathbf{\pi}^{\rm{div}} \mathbf{u}_h). \label{eq:thm_gradient_robust_eq3}
  \end{equation}
  Now $\pi^{L^2} \phi \in Q_h$ and $\mathbf{u}_h \in \mathbf{V}_h$
  hence, by~\eqref{eq:comm_prop}, it holds
  \begin{equation*}
    (\pi^{L^2} \phi, \nabla \cdot \mathbf{\pi}^{\rm{div}} \mathbf{u}_h) = ( \pi^{L^2} \phi, \nabla \cdot \mathbf{u}_h) . 
  \end{equation*}
  Filling this into~\eqref{eq:thm_gradient_robust_eq3} gives
  \begin{equation}
    \begin{aligned}
      2 \mu \left( \varepsilon(\mathbf{u}_h), \varepsilon(\mathbf{u}_h)\right) + \lambda\left(\pi^{L^2} \nabla \cdot \mathbf{u}_h , \pi^{L^2} \nabla \cdot \mathbf{u}_h \right)
      & = -  \left(\pi^{L^2}\phi,\pi^{L^2} \nabla \cdot \mathbf{u}_h\right).  \label{eq:thm_gradient_robust_eq5}
    \end{aligned}
  \end{equation}
  Using Cauchy-Schwartz inequality, we get
  \begin{equation}
    2\mu\|\varepsilon(\mathbf{u}_h)\|^2_0 + \lambda\|\pi^{L^2}
    \nabla \cdot \mathbf{u}_h\|_0^2 \le \|\pi^{L^2} \phi\|_0
    \|\pi^{L^2}\nabla \cdot \mathbf{u}_h\|_0
    \le  \|\phi\|_0\|\pi^{L^2}\nabla \cdot \mathbf{u}_h\|_0.\label{eq:thm_gradient_robust_eq6}
  \end{equation}

  Now, to estimate the $H^1$-norm of $\mathbf{u}_h$, we notice that
  by the choice of $\mathbf{f}$
  and~\eqref{eq:compatibility_divergence}, testing the first
  equation in~\eqref{eq:modified_weak_form} with a function
  $\mathbf{v}_h \in \mathbf{V}_h^0$ yields
  \[
  a(\mathbf{u}_h,\mathbf{v_h}) =
  -b(p_h,\mathbf{v}_h)-(\phi,\nabla\cdot
  \mathbf{\pi}^{\rm{div}}\mathbf{v}_h) = 0
  \]
  and thus $\mathbf{u}_h \in \mathbf{V}_h^\bot$. Hence by,
  e.g.,~\cite[Lemma~3.58]{John:2016} it holds
  \begin{equation}\label{eq:thm_gradient_robust_eq7}
    \|\mathbf{u}_h\|_1 \le \frac{c}{\beta}\|\pi^{L^2}\nabla \cdot \mathbf{u}_h\|_0
  \end{equation}
  with the inf-sup constant $\beta$ from~\eqref{eq:inf-sup},
  since $\pi^{L^2}\nabla \cdot \mathbf{u}_h \in Q_h$.
  
  Using Korn's inequality,~\eqref{eq:thm_gradient_robust_eq6}, and~\eqref{eq:thm_gradient_robust_eq7}, we get 
  \begin{equation}
    \begin{aligned}
      (\mu+\lambda) \|\mathbf{u}_h\|^2_1 &\le
      c\mu\|\varepsilon(\mathbf{u}_h)\|_0^2 + \frac{\lambda
        c}{\beta}\|\pi^{L^2}\nabla \cdot \mathbf{u}_h\|_1^2 \\
      &\le c\|\phi\|_0\|\mathbf{u}_h\|_1,
    \end{aligned}
  \end{equation}
  and thus the assertion is shown.
\end{proof}
\begin{theorem}\label{thm:gradient_robust_error}
  Let Assumptions~\ref{ass:fespace} and~\ref{ass:1} be satisfied.
  Then the solutions $(\mathbf{u}, p) \in \mathbf{V} \times Q$,
  of the problem~\eqref{eq:elasticity_mixed} and
  $(\mathbf{u}_h, p_h) \in \mathbf{V}_h \times Q_h$
  of~\eqref{eq:modified_weak_form} satisfy the error estimate
  \begin{equation}\label{eq:main_ineq}
    \|\mathbf{u} - \mathbf{u}_h\|_1 \le c\, h^k \left(1 + \sqrt{\frac{\mu}{\lambda}}\right)\|\mathbf{u}\|_{k+1} +
    c\frac{h^k}{\mu\lambda}\|p\|_k,
  \end{equation}
  provided the regularity $(\mathbf{u}, p) \in
  H^{k+1}(\Omega;\mathbb{R}^d)\times H^k(\Omega)$
  is given.
\end{theorem}

\begin{proof}
  As in the proof of Theorem~\ref{thm:thm1_infty}, we could split the error
  \[
  (\mathbf{u}-\mathbf{u}_h,p-p_h) =
  (\mathbf{u}-\mathbf{v}_h,p-q_h)  + (\mathbf{v}_h-\mathbf{u}_h,q_h-p_h) 
  \]
  with arbitrary $\mathbf{v}_h \in \mathbf{V}_h$ and $q_h \in
  Q_h$. However, as it will turn out to be useful, we will select
  $q_h = \pi^{L^2} p$ and $\mathbf{v}_h$ as the elasticity projection of $\mathbf{u}$,
  i.e., satisfying the following equation
  \begin{align}\label{eq:elasticity_projection}
    \left( \varepsilon(\mathbf{v}_h), \varepsilon(\boldsymbol{\varphi}_h)\right) + 
    b(\tilde{p}_h, \boldsymbol{\varphi}_h) 
    & = \left(\varepsilon(\mathbf{u}), \varepsilon(\boldsymbol{\varphi}_h)\right) 
    && \forall \boldsymbol{\varphi}_h \in \mathbf{V}_h, \\
    b(s_h,\mathbf{v}_h) & = b(s_h, \mathbf{u}) && \forall s_h \in Q_h. \nonumber
  \end{align}
  Clearly, the solution to the continuous counterpart is
  $(\mathbf{v},\tilde{p}) = (\mathbf{u},0)$.  
  Since the above equation is uniquely solvable, see,
  e.g.~\cite[Theorem~4.2.3]{BoffiBrezziFortin:2013}, we have the
  orthogonality $b(\mathbf{\pi}^{L^2} p - p_h, \mathbf{u} - \mathbf{v}_h)  = 0$ 
  and the approximation error satisfies, e.g.,~\cite[Theorem~5.2.2]{BoffiBrezziFortin:2013}.
  \begin{equation}
    \|\mathbf{u} - \mathbf{v}_h\|_1 + \|\tilde{p} - \tilde{p}_h\|
    \le  c\inf\limits_{\boldsymbol{\varphi}_h \in \mathbf{V}_h}
    \|\mathbf{u} - \boldsymbol{\varphi}_h\|_1  + c\inf\limits_{s_h \in Q_h}\|0- q_h\|,
  \end{equation}
  which gives 
  \begin{equation}
    \|\mathbf{u} - \mathbf{v}_h\|_1 \le c\inf\limits_{\mathbf{\varphi}_h \in \mathbf{V}_h }
    \|\mathbf{u} - \mathbf{\varphi}_h\|_1 \label{eq:main_thm_u_norm}
  \end{equation}
  
  Due to the interpolation estimates in
  Assumption~\ref{ass:fespace}, we are left with bounding 
  $\mathbf{w}_h  = \mathbf{u}_h - \mathbf{v}_h \in \mathbf{V}_h$ and
  $r_h = p_h - q_h \in Q_h$. We split $\mathbf{w}_h = \mathbf{w}_h^0
  + \mathbf{w}_h^\bot \in \mathbf{V}_h^0 \oplus \mathbf{V}_h^\bot$.
  By definition of the bilinear forms $a$ and $b$,
  i.e.,~\eqref{eq:bilinear_a} and~\eqref{eq:bilinear_b}, and the
  first line in~\eqref{eq:modified_weak_form} and~\eqref{eq:elasticity_mixed},
  the remainder $\mathbf{w}_h$ and $r_h$ satisfy, for any discrete
  function $ \boldsymbol{\varphi}_h \in \mathbf{V}_h$,
  \begin{equation}\label{eq:thm5_proof1}
    \begin{aligned}
      a(\mathbf{w}_h, \boldsymbol{\varphi}_h) &+ b(r_h,\boldsymbol{\varphi}_h)  
      = a(\mathbf{u}_h - \mathbf{v}_h, \boldsymbol{\varphi}_h) + b(p_h - q_h,\boldsymbol{\varphi}_h) \\
      & = (\mathbf{f}, \mathbf{\pi}^{\rm{div}}\boldsymbol{\varphi}_h -  \boldsymbol{\varphi}_h) +
      a(\mathbf{u} - \mathbf{v}_h, \boldsymbol{\varphi}_h) + b(p - q_h, \boldsymbol{\varphi}_h).
    \end{aligned}
  \end{equation}
  Analogously, from the second line
  in~\eqref{eq:modified_weak_form} and~\eqref{eq:elasticity_mixed},
  we get for arbitrary $s_h \in Q_h$
  \begin{equation}\label{eq:gradient_robust_error_proof2}
    \begin{aligned}
      b(s_h,\mathbf{w}_h) - \frac{1}{\lambda} ( r_h, s_h) & = b(s_h,\mathbf{u}_h - \mathbf{v}_h) - \frac{1}{\lambda}(p_h - q_h, s_h) \\
      & = b(s_h,\mathbf{u}_h) - \frac{1}{\lambda}(p_h, s_h) - \bigl(b(s_h,\mathbf{v}_h) - \frac{1}{\lambda}(q_h, s_h)\bigr) \\
      & = b(s_h,\mathbf{u} - \mathbf{v}_h) - \frac{1}{\lambda}(p - q_h, s_h). 
    \end{aligned}
  \end{equation}
  Testing~\eqref{eq:thm5_proof1}
  and~\eqref{eq:gradient_robust_error_proof2}
  with $\mathbf{\varphi}_h=\mathbf{w_h}$ and $s_h=r_h$ we get
  \begin{equation} \label{eq:gradient_robust_error_proof3}
    \begin{aligned}
      c\mu\|\mathbf{w}_h\|^2_1 + \frac{1}{\lambda} \|r_h\|^2
      &\le a(\mathbf{w}_h, \mathbf{w}_h) + \frac{1}{\lambda} (r_h, r_h)\\        
      &= a(\mathbf{w}_h, \mathbf{w}_h) + b(r_h,\mathbf{w}_h) - b(r_h,\mathbf{w}_h)
      + \frac{1}{\lambda}(r_h, r_h) \\
      & = (\mathbf{f}, \mathbf{\pi}^{\rm{div}}\mathbf{w}_h -
      \mathbf{w}_h)
      + a(\mathbf{u} - \mathbf{v}_h, \mathbf{w}_h) \\
      & \;\;\;\; +b(p - q_h,\mathbf{w}_h) - b(r_h,\mathbf{u} - \mathbf{v}_h)
      - \frac{1}{\lambda}(p - q_h, r_h).
    \end{aligned}
  \end{equation}
  Using~\eqref{eq:lemma2} and~\eqref{eq:elasticity_mixed},
  we obtain a bound on
  $(\mathbf{f}, \mathbf{\pi}^{\rm{div}}\mathbf{w}_h - \mathbf{w}_h)$
  as follows
  \begin{equation*}
    \begin{aligned}
      ( \mathbf{f}, \mathbf{\pi}^{\rm{div}}\mathbf{w}_h &- \mathbf{w}_h)   = -2\mu(\nabla \cdot \varepsilon(\mathbf{u}), \mathbf{\pi}^{\rm{div}}\mathbf{w}_h -\mathbf{w}_h) - 
      (\nabla p , \mathbf{\pi}^{\rm{div}}\mathbf{w}_h - \mathbf{w}_h) \\
      & = -2\mu(\nabla \cdot \varepsilon(\mathbf{u}),
      \mathbf{\pi}^{\rm{div}}\mathbf{w}_h) -
      2\mu\left(\varepsilon(\mathbf{u}),
      \varepsilon(\mathbf{w}_h)\right) + b(p,\mathbf{\pi}^{\rm{div}}\mathbf{w}_h - \mathbf{w}_h) \\
      & \le c \sum\limits_{T \in \mathcal{T}_h}
      h_T^k\lvert\mathbf{u}\rvert_{k+1,T}
      \lvert\mathbf{w}_h\rvert_{1,T}
      + b(p,\mathbf{\pi}^{\rm{div}}\mathbf{w}_h - \mathbf{w}_h) \\
      & \le 2\mu c \left( \sum\limits_{T \in \mathcal{T}_h}
      h^{2k}_T\lvert\mathbf{u}\rvert^2_{k+1,T}\right)^{\frac{1}{2}}\lvert\mathbf{w}_h\rvert_1
      + b(p,\mathbf{\pi}^{\rm{div}}\mathbf{w}_h - \mathbf{w}_h). 
    \end{aligned}
  \end{equation*}
  Substituting this in~\eqref{eq:gradient_robust_error_proof3}, we get
  \begin{equation}\label{eq:gradient_robust_error_proof5}
    \begin{aligned}
      c\mu\|\mathbf{w}_h\|^2_1 &+ \frac{1}{\lambda} \|r_h\|^2 
      \le 2\mu c \left( \sum\limits_{T \in \mathcal{T}_h}  h^{2k}_T\lvert\mathbf{u}\rvert^2_{k+1,T}\right)^{\frac{1}{2}}\lvert\mathbf{w}_h\rvert_1  \\
      &+ \Bigl(b(p,\mathbf{\pi}^{\rm{div}}\mathbf{w}_h - \mathbf{w}_h) + b(p - q_h,\mathbf{w}_h) - b(r_h,\mathbf{u} - \mathbf{v}_h)\Bigr) \\
      &+\Bigl(a(\mathbf{u}- \mathbf{v}_h, \mathbf{w}_h) - \frac{1}{\lambda}(p - q_h, r_h)\Bigr). 
    \end{aligned}
  \end{equation}
  The last line can be estimated as
  \begin{equation*}
    a(\mathbf{u}- \mathbf{v}_h, \mathbf{w}_h) - \frac{1}{\lambda}(p - q_h, r_h) \le \frac{c\mu}{2}\|\mathbf{u} - \mathbf{v}_h\|_1^2  + 
    \frac{c\mu}{2} \|\mathbf{w}_h\|^2_1 + \frac{1}{2\lambda}\|p - q_h\|^2 + 
    \frac{1}{2\lambda} \|r_h\|^2.
  \end{equation*}
  From~\eqref{eq:comm_prop},
  we have that $b(q_h,\mathbf{\pi}^{\rm{div}}\mathbf{w}_h -
  \mathbf{w}_h) = 0$.
  Hence the second line in~\eqref{eq:gradient_robust_error_proof5} becomes
  \begin{align*}
    b(p,\mathbf{\pi}^{\rm{div}}\mathbf{w}_h - \mathbf{w}_h) + & \, b(p - q_h,\mathbf{w}_h) - b(r_h,\mathbf{u} - \mathbf{v}_h)  \\
    & =  b(p - q_h,\mathbf{\pi}^{\rm{div}}\mathbf{w}_h - \mathbf{w}_h) + b(p-q_h,\mathbf{w}_h) - b(r_h,\mathbf{u} - \mathbf{v}_h) \\
    & =  b(p - q_h,\mathbf{\pi}^{\rm{div}}\mathbf{w}_h) -
    b(p_h - q_h,\mathbf{u} - \mathbf{v}_h)\\
    &= b(\pi^{L^2}p - q_h,\mathbf{\pi}^{\rm{div}}\mathbf{w}_h) -
    b(p_h - q_h,\mathbf{u} - \mathbf{v}_h)\\
    &= b(\pi^{L^2}p - q_h,\mathbf{w}_h) -
    b(p_h - q_h,\mathbf{u} - \mathbf{v}_h)\\
  \end{align*}
  where we used the properties of the $L^2$ projection $\pi^{L^2}$,
  the commutative diagram~\eqref{eq:comm_prop} and $\nabla
  \cdot \mathcal M_h \subset Q_h$.
  Now, we utilize the choice $q_h = \pi^{L^2} p$ to further simplify 
  the representation of the second line
  in~\eqref{eq:gradient_robust_error_proof5} to be
  \begin{align*}
    b(p,\mathbf{\pi}^{\rm{div}}\mathbf{w}_h - \mathbf{w}_h) + & \, b(p - q_h,\mathbf{w}_h) - b(r_h,\mathbf{u} - \mathbf{v}_h) \\
    & = b(\pi^{L^2}p - q_h,\mathbf{w}_h) - b(p_h - q_h,\mathbf{u}-\mathbf{u}_h) \\
    & = b(\pi^{L^2} p - p_h,\mathbf{u} - \mathbf{v}_h)\\
    &= 0
  \end{align*}
  by our choice of $\mathbf{v}_h$. This provides the bound
  \begin{equation}\label{eq:gradient_robust_error_proof_intermediate}
    \begin{aligned}
      \frac{c\mu}{2}\|\mathbf{w}_h\|^2_1 &+ \frac{1}{2\lambda} \|r_h\|^2 
      \le 2\mu c \left( \sum\limits_{T \in \mathcal{T}_h}  h^{2k}_T\lvert\mathbf{u}\rvert^2_{k+1,T}\right)^{\frac{1}{2}}\lvert\mathbf{w}_h\rvert_1  \\
      &+  \frac{c\mu}{2}\|\mathbf{u} - \mathbf{v}_h\|_1^2  + 
      \frac{1}{2\lambda}\|p - q_h\|^2.
    \end{aligned}
  \end{equation}

  Of course~\eqref{eq:gradient_robust_error_proof_intermediate} provides a
  bound on $\mathbf{w}_h$ but as it is suboptimal we continue by splitting 
  $\mathbf{w}_h = \mathbf{w}_h^0+\mathbf{w}_h^{\bot}$.
  
  We first bound $\|\mathbf{w}_h^0\|_1$.
  Consider $c\mu \|\mathbf{w}_h^0\|_1$ and using that
  $a(\mathbf{w}_h^\bot, \mathbf{w}_h^0) = 0$, we have,
  using~\eqref{eq:Vh0},~\eqref{eq:thm5_proof1}, and the choice of
  $\mathbf{v}_h$ as elasticity projection that 
  \begin{equation*}
    \begin{aligned}
      c\mu\|\mathbf{w}_h^0\|_1^2 & \le a(\mathbf{w}_h^0, \mathbf{w}_h^0) = a(\mathbf{w}_h, \mathbf{w}_h^0) = 
      a(\mathbf{w}_h, \mathbf{w}_h^0) + b(\mathbf{r}_h, \mathbf{w}_h^0) 	\\
      & = (f, \mathbf{\pi}^{\rm{div}}\boldsymbol{w}_h^0 -  \boldsymbol{w}_h^0)  \\
      & \le \left(-2\nabla\mu \cdot \varepsilon(\mathbf{u}) + \nabla p, 
      \mathbf{\pi}^{\rm{div}}\mathbf{w}_h^0 -  \mathbf{w}_h^0 \right)  \\
      & \le \left(-2\nabla\mu \cdot \varepsilon(\mathbf{u}) , 
      \mathbf{\pi}^{\rm{div}}\mathbf{w}_h^0 -  \mathbf{w}_h^0 \right) 
      + \left(\nabla p , \mathbf{\pi}^{\rm{div}}\mathbf{w}_h^0 -  \mathbf{w}_h^0 \right)\\
      & \le  \left(-2\nabla \cdot \mu\varepsilon(\mathbf{u}), 
      \mathbf{\pi}^{\rm{div}} \mathbf{w}_h^0\right) - 
      \mu \left(\varepsilon(\mathbf{u}), \varepsilon(\mathbf{w}_h^0)\right) \\
      & \le \mu \left(-2\nabla \cdot \varepsilon(\mathbf{u}), \mathbf{\pi}^{\rm{div}} 
      \mathbf{w}_h^0\right) - \mu \left(\varepsilon(\mathbf{u}), 
      \varepsilon(\mathbf{w}_h^0)\right). \\
    \end{aligned}
  \end{equation*}
  Thus, by Lemma~\ref{lem:consistency}, we conclude
  \begin{equation*}
    c\mu\|\mathbf{w}_h^0\|_1^2 \le \mu c \sum\limits_{T \in \mathcal{T}_h} h_T^k\|\mathbf{u}\|_{k+1,T} 
    \vert\mathbf{w}_h^0\vert_{1,T}
    \le c\mu h^k \|\mathbf{u}\|_{k+1} \|\mathbf{w}_h^0\|_1 \\
  \end{equation*}
  and hence
  \begin{equation}\label{eq:w0_bound}
    \|\mathbf{w}_h^0\|_1   \le ch^k \|\mathbf{u}\|_{k+1} .
  \end{equation}
  For $\|\mathbf{w}_h^\bot\|_1$, we utilize $\mathbf{w}_h^\bot \in
  \mathbf{V}_h^\bot$, i.e.,
  \[
  \left(\nabla \cdot \mathbf{w}_h, q_h\right) =
  \left(\nabla \cdot \mathbf{w}_h^\bot, q_h\right) \qquad \forall
  q_h \in Q_h
  \]
  meaning
  \[
  \mathbf{\pi}^{L^2} \nabla \cdot \mathbf{w}_h =
  \mathbf{\pi}^{L^2}\nabla \cdot \mathbf{w}_h^\bot.
  \]

  Using~\cite[Lemma~3.58]{John:2016}, we  get
  \begin{equation*}
    \begin{aligned}
      \|\mathbf{w}_h^\bot\|_1 & \le \frac{c}{\beta}
      \|\mathbf{\pi}^{L^2} \left(\nabla \cdot \mathbf{w}_h \right)\|_0 \\
      & \le \frac{c}{\beta} \|\mathbf{\pi}^{L^2} \nabla \cdot \mathbf{u}_h - 
      \mathbf{\pi}^{L^2}\nabla \cdot \mathbf{v}_h\|_0 \\
      & \le \frac{c}{\beta}\Bigl\|\frac{p_h}{\lambda}  - 
      \mathbf{\pi}^{L^2}\nabla \cdot
      \mathbf{v}_h\Bigr\|_0
    \end{aligned}
  \end{equation*}
  from the definition of $\mathbf{v}_h$ as elasticity projection.
  Hence, noting that $\nabla \cdot u = \frac{1}{\lambda} p$, we obtain
  \begin{equation*}
    \|\mathbf{w}_h^\bot\|_1 \le \frac{c}{\beta \lambda}\|p_h -
    q_h\|_0 = \frac{c}{\beta \lambda} \| r_h\|_0. 
  \end{equation*}

  With this, we conclude
  from~\eqref{eq:gradient_robust_error_proof_intermediate}
  \begin{equation*}
    \begin{aligned}
     \|r_h\|_0^2 & \le \lambda (\mu\|\mathbf{w}_h\|_1^2 + \frac{1}{\lambda}\|r_h\|_0^2) \\
     & \le c\mu\lambda h^{2k}\|\mathbf{u}\|^2_{k+1} + 
      c\|p - q_h\|_0^2 
    \end{aligned}
  \end{equation*}
  and thus   
 \begin{equation} \label{eq:wbot_bound}
   \begin{aligned}
     \|\mathbf{w}_h^\bot\|_1 &\le
     \frac{c}{\lambda} \|r_h\|_0\\
     & \le c\sqrt{\frac{\mu}{\lambda}}h^k\|\mathbf{u}\|_{k+1} + 
      \frac{c}{\lambda}\|p - q_h\|_0.
   \end{aligned}
 \end{equation}

  Now, we can bound $\|\mathbf{w}_h\|_1$ using \eqref{eq:w0_bound} and \eqref{eq:wbot_bound}
  \begin{equation}\label{eq:wh_bound}
    \begin{aligned}
      \|\mathbf{w}_h\|_1 & \le \|\mathbf{w}_h^0\|_1 + \|\mathbf{w}_h^\bot\|_1 \\
      & \le ch^k\|\mathbf{u}\|_{k+1} + \frac{c}{\beta\lambda}\|r_h\|_0 \\
      & \le ch^k\|\mathbf{u}\|_{k+1} + \frac{c}{\beta}\sqrt{\frac{\mu}{\lambda}}h^k\|\mathbf{u}\|_{k+1}
      + \frac{c}{\lambda}\|p -q_h\|_0\\
      &\le c\left(1 + \sqrt{\frac{\mu}{\lambda}}\right)h^k\|\mathbf{u}\|_{k+1} + 
      \frac{c}{\lambda}\|p - q_h\|_0.
    \end{aligned}
  \end{equation} 
  Finally, we arrive at the desired bound 
  \begin{equation}\label{eq:u-u_h_error}
    \begin{aligned}
      \|\mathbf{u} - \mathbf{u}_h\|_1 & \le \|\mathbf{u} - \mathbf{v}_h\|_1 + \|\mathbf{w}_h\|_1 \\
      & \le c\left(1 + \sqrt{\frac{\mu}{\lambda}}\right) h^k\|\mathbf{u}\|_{k+1}  + \frac{c}{\lambda}h^k\|p\|_k
    \end{aligned}
  \end{equation}
  by definition of $q_h$
  and Assumption~\ref{ass:fespace}.
\end{proof}

\section{Numerical Results}\label{sec:numerics}
For our computation, we use DOpElib~\cite{DOpElib} based
on the deal.II~\cite{dealII93} finite element library.
First, we present an example for incompressible materials.
\begin{example}\label{ex:ex2}
  For the first numerical example, we consider a small variation of Example $5.1$ in~\cite{Linke:2014}, 
  where the displacement and pressure is given as
  \begin{equation}\label{eq:ex1-solution}
    \mathbf{u}(x, y) = \begin{bmatrix}
      200 x^2 (1-x)^2y(1-y)(1-2y) \\
      -200 y^2(10y)^2x(1-x)(1-2x)
    \end{bmatrix}
  \end{equation}  
  \begin{equation}
    p(x,y) = -10\left(x - \frac{1}{2}\right)^3y^2 + (1-x)^3\left(y-\frac{1}{2}\right)^3 + \frac{1}{8}.
  \end{equation}
  for the incompressible linear elasticity equation
  \begin{equation}
    \begin{aligned}
      -2\mu \nabla  \cdot  \varepsilon(\mathbf{u}) + \nabla p = \mathbf{f}, \\
      \nabla \cdot \mathbf{u} = 0
    \end{aligned}
  \end{equation}
  with thus defined $\mathbf{f}$.
\end{example}
\begin{figure}
  \centering
  \includegraphics[width = 0.5\textwidth]{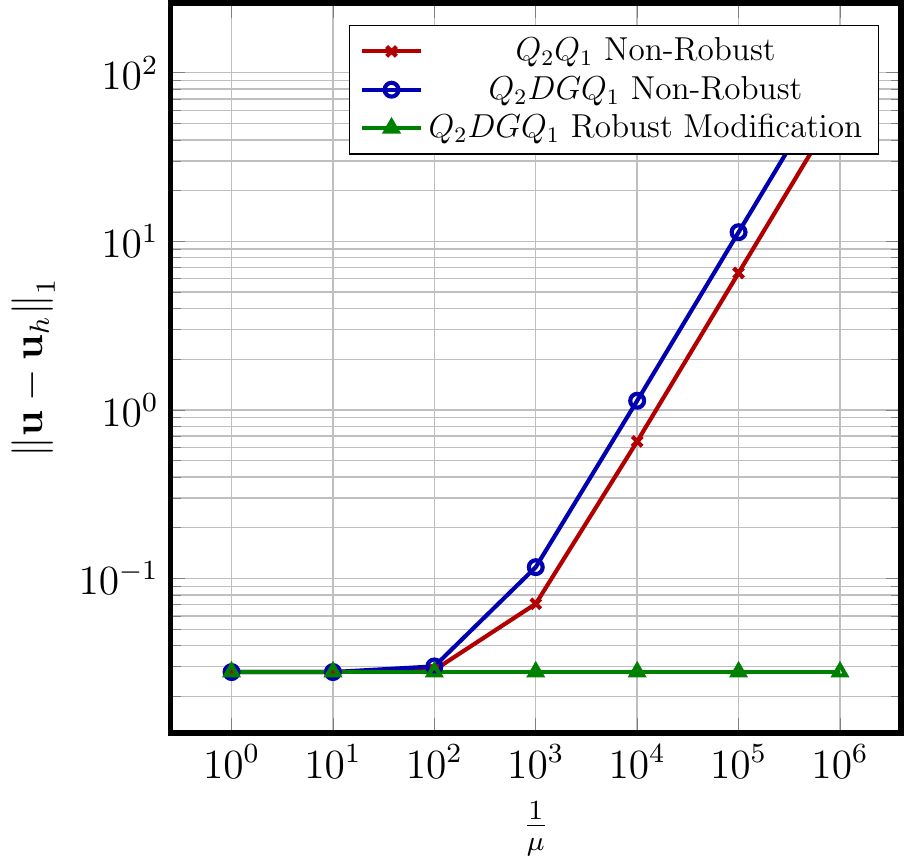}
  \caption{\small{Comparing displacement error in $H^1$ norm vs. $\frac{1}{\mu}$ for Example~\ref{ex:ex2} with and without 
      gradient robust modification for $h = 2^{-3}$} }\label{fig:ex2}
\end{figure}
Comparing~\eqref{eq:disp_error} with Figure~\ref{fig:ex2},
we notice that the $H^1$-norm displacement error without 
interpolation grows linearly w.r.t $\frac{1}{\mu}$ as predicted
due to the appearance of the 
pressure term $\frac{1}{\mu} \inf\limits_{q_h \in Q_h} \|p -
q_h\|_0$ in~\eqref{eq:disp_error}.
For the gradient robust modification employing the interpolation
onto the $\mathcal{BDM}$ finite element space, the error is
independent of $\mu$, highlighting the prediction of
Theorem~\ref{thm:thm1_infty}.

For future examples, we consider nearly incompressible materials given by equation~\eqref{eq:elasticity_mixed}.
\begin{example}\label{ex:ex3}
  For the second numerical example, we set the right hand side $f = \nabla \phi; \phi = x^6 + y^6$ in
  equation~\eqref{eq:elasticity_mixed}, as
  3  in~\cite[Example~2]{FuGuoshengLehrenfeldChristophLinkeStrechenbach:2020}.
\end{example}

From Lemma~\ref{lem:u0}, the solution for Example~\ref{ex:ex3} is given as
$\mathbf{u} = 0$ and $p = x^6 + y^6$. From
equation~\eqref{eq:thm4_error}, we have the bound
\[
\|\mathbf{u}_h\|_1 \le \frac{c}{\lambda + \mu} \| \phi\|_0
\]
on the discrete function for a gradient robust discretization.
For $\mu = 10^{-5}$, we have $\lambda + \mu \approx \lambda, \forall \lambda \geq 1$. Hence, we see
a green line with positive slope in Figure~\ref{fig:ex3_lambda} for
the gradient robust method, while the non robust method shows an
almost constant $\|\mathbf{u}_h\|_1 \ne 0$. However, for $\lambda = 10^{5}$
we have $\frac{1}{\lambda + \mu} \approx c$(constant) $\forall 0 < \mu \le 1$, which is seen in the flat green line 
in Figure~\ref{fig:ex3_mu}. \\
For non-gradient robust methods, we have 
\[
\|\mathbf{u}_h\|_1 \le
\frac{c}{\mu}\left(\frac{1}{\lambda} + 1\right)\|\phi\|_0
\]
from equation~\eqref{eq:disp_error}. For $\mu = 10^{-5}$, the term $\left( \frac{1}{\lambda} + 1\right) \to 1$ as $\lambda \to \infty$.
The same is shown by the flat red line in Figure~\ref{fig:ex3_lambda}. However, for $\lambda= 10^{-5}$, we have 
$\|\mathbf{u}\|_1 \le \frac{c}{\mu}$. Which is shown by the red line with negative slope in Figure~\ref{fig:ex3_mu}.

\begin{figure}
  \begin{minipage}{\textwidth}
    \begin{subfigure}{.5\textwidth}
      \centering
      \includegraphics[width = \textwidth]{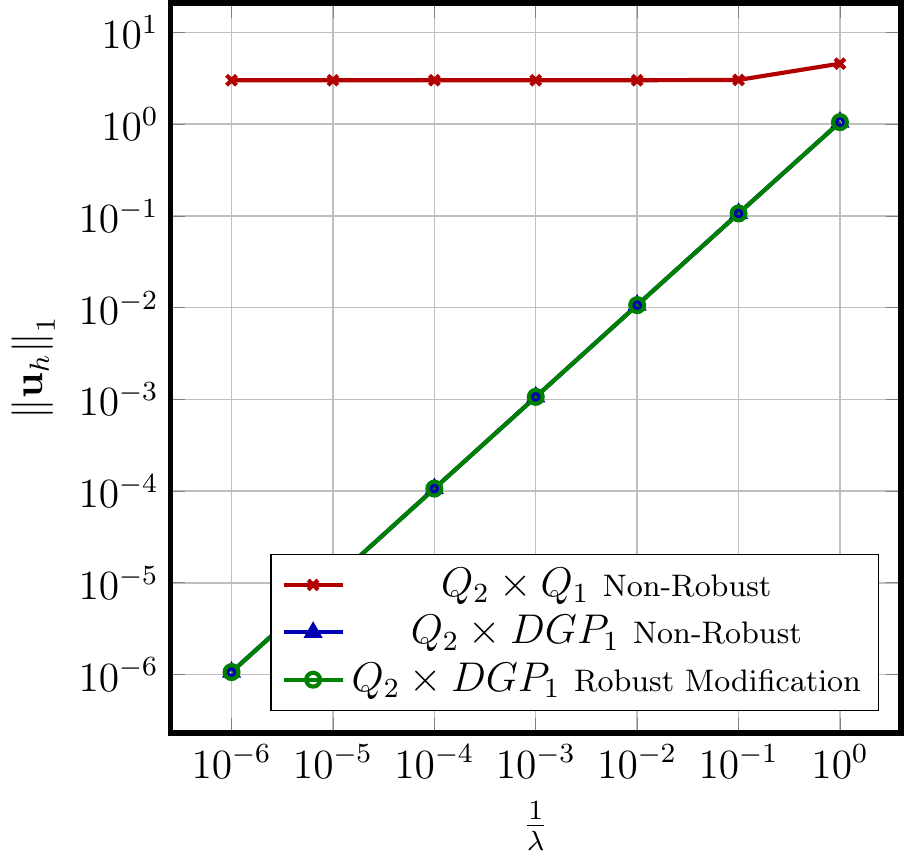}
      \caption{\small{$\|\mathbf{u}_h\|_1$  vs. $\frac{1}{\lambda}$ with $\mu = 10^{-5}$}}\label{fig:ex3_lambda}
    \end{subfigure}%
    \begin{subfigure}{0.5\textwidth}
      \centering
      \includegraphics[width = \textwidth]{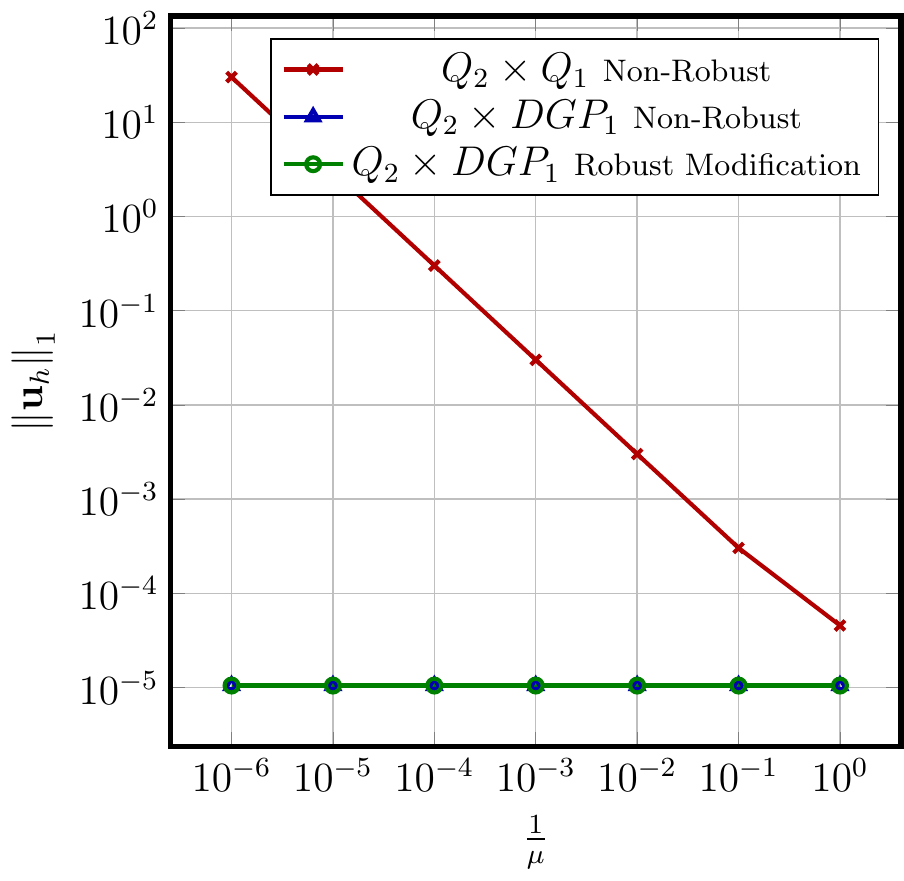}
      \caption{\small{$\|\mathbf{u}_h\|_1$ vs. $\frac{1}{\mu}$ with $\lambda = 10^{-5}$}} \label{fig:ex3_mu}
    \end{subfigure}
  \end{minipage}
  \caption{\small{Comparing displacement error in $H^1$ norm for 
      Example~\ref{ex:ex3} with and without gradient robust
      modification for $h = 2^{-3}$}}\label{fig:ex3}
\end{figure}
It should be noted in this example, that the line for the
non-gradient robust $Q_2\times \operatorname{DGP}_1$ method coincides
with the gradient robust modification. However, this effect is due to
a too simple pressure. That indeed, the standard
$Q_2\times \operatorname{DGP}_1$ method is not gradient robust is
shown in the following example.

\begin{example}\label{ex:ex3b}
  For the third numerical example, we consider the right hand side $f = \nabla \phi; \phi = -10(x-0.5)^3y^2 + (1-x)^3(y-0.5)^3 - 1/8$ 
  in Example~\ref{ex:ex3}.
\end{example} 

Figure~\ref{fig:ex3b} shows our previous statement, that
Example~\ref{ex:ex3} had a pressure which is too simple to show the missing
gradient robustness of the standard $Q_2 \times \operatorname{DGP}_1$
discretization. Indeed, in this example, both $Q_2\times Q_1$ and
$Q_2 \times \operatorname{DGP}_1$ discretization show the undesirable
blowup for $\mu \rightarrow 0$ and the constant value as $\lambda
\rightarrow \infty$, while the gradient robust modification shows the
desired convergence.

\begin{figure}
  \begin{minipage}{\textwidth}
    \begin{subfigure}{.5\textwidth}
      \centering
      \includegraphics[width = \textwidth]{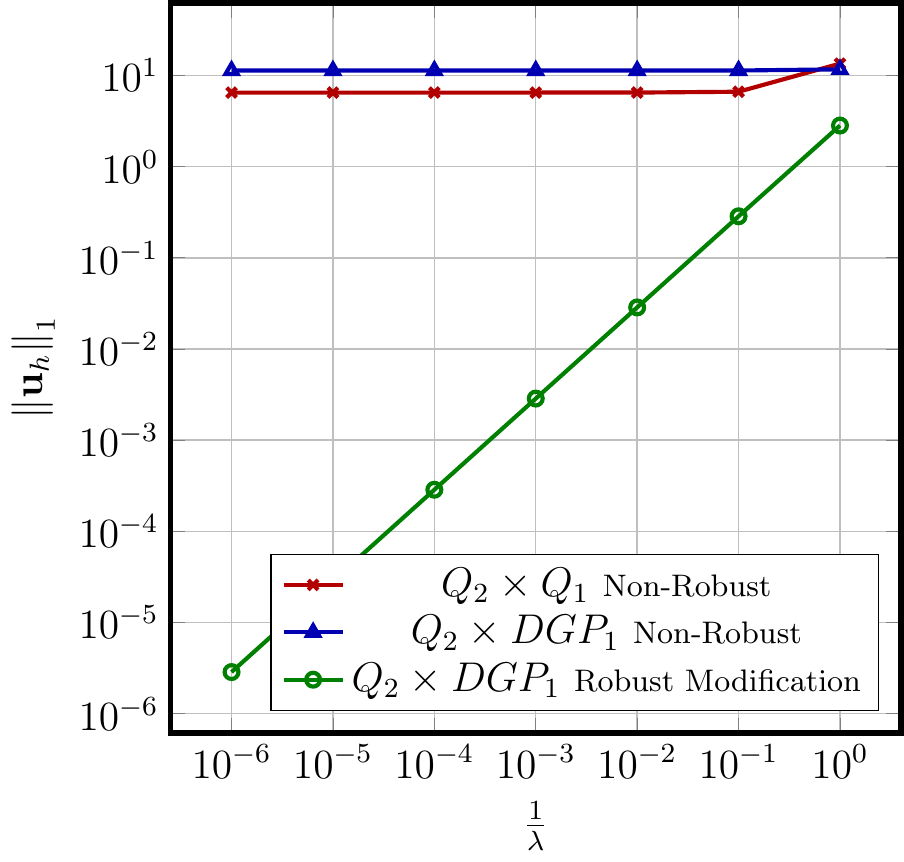}
      \caption{\small{$\|\mathbf{u}_h\|_1$  vs. $\frac{1}{\lambda}$ with $\mu = 10^{-5}$}}\label{fig:ex3b_lambda}
    \end{subfigure}%
    \begin{subfigure}{0.5\textwidth}
      \centering
      \includegraphics[width = \textwidth]{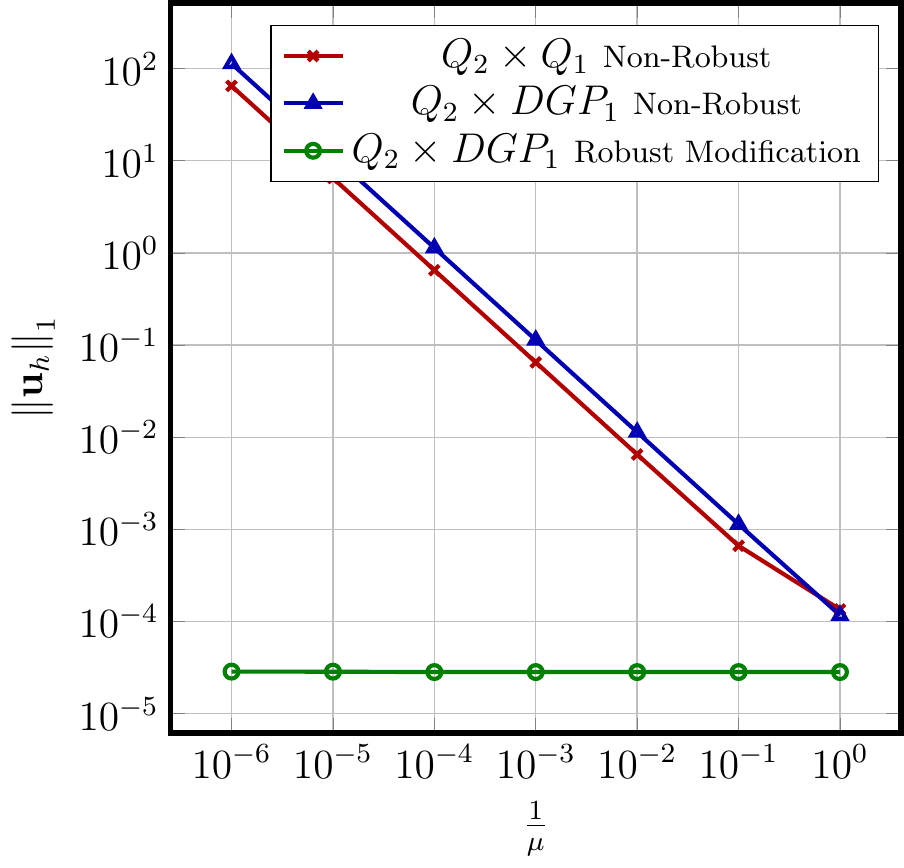}
      \caption{\small{$\|\mathbf{u}_h\|_1$ vs. $\frac{1}{\mu}$ with $\lambda = 10^{-5}$}} \label{fig:ex3b_mu}
    \end{subfigure}
  \end{minipage}
  \caption{\small{Comparing displacement error in $H^1$ norm for 
      Example~\ref{ex:ex3b} with and without gradient robust
      modification for $h = 2^{-3}$}}\label{fig:ex3b}
\end{figure}

\begin{example}\label{ex:ex4}
  For the fourth example, we consider the nearly incompressible case
  $(\lambda \neq \infty)$, i.e., 
  \begin{equation}
    \begin{aligned}
      -2\mu \nabla  \cdot  \varepsilon(\mathbf{u}) + \nabla p = \mathbf{f}, \\
      \nabla \cdot \mathbf{u} - \frac{1}{\lambda}p= 0
    \end{aligned}
  \end{equation}
  where we use the same $f$ as in Example~\ref{ex:ex2}. 
\end{example}

In this example, for $\lambda = \infty$, the solution $\mathbf{u}^{\infty}$
is known, i.e., it is given in~\eqref{eq:ex1-solution}.
We denote the solution, for $\lambda \neq \infty$, as $\mathbf{u}^\lambda$.
We compute the error $\| \mathbf{u}^\infty -
\mathbf{u}_h^\lambda\|$ in our numerical results,
where $\mathbf{u}_h^\lambda$ is the discrete approximated solution for a given value of $\lambda$. 
Since, Theorem~\ref{thm:gradient_robust_error} provides an
estimate, for $\| \mathbf{u}^\lambda - \mathbf{u}^\lambda_h\|$,
only, we use the triangle inequality to get
\begin{align}
  \|\mathbf{u}^\infty - \mathbf{u}^\lambda_h\|_1 &\le \|\mathbf{u}^\infty - \mathbf{u}^\lambda\|_1 + 
  \|\mathbf{u}^\lambda - \mathbf{u}^\lambda_h\|_1, \\
  & \le \|\mathbf{u}^\infty - \mathbf{u}^\lambda\|_1 + 
  c\left(1+\sqrt{\frac{\mu}{\lambda}}\right) h^2\|\mathbf{u}\|_3 + \frac{ch^2}{\lambda}\|p\|_2. \label{eq:ineq}
\end{align}

Figure~\ref{fig:ex4_mu} follows the same pattern as Figure~\ref{fig:ex3b_mu}. However, there is a slight difference
between Figures~\ref{fig:ex4_lambda} and~\ref{fig:ex3b_lambda}, which can be explained by \eqref{eq:ineq}.
When $\lambda \to \infty$, we have 
\[  \|\mathbf{u}^\infty - \mathbf{u}^\lambda\|_1 + ch^2\|\mathbf{u}_h\|_3  \gg 
ch^2 \sqrt{\frac{\mu}{\lambda}}\|\mathbf{u}\|_3 + \frac{ch^2}{\lambda}\|p\|_2. \]
So the estimate on $\|\mathbf{u}^\infty - \mathbf{u}^\lambda_h\|_1$ converges to 
$\|\mathbf{u}^\infty - \mathbf{u}^\lambda\|_1 \ne 0$ and thus saturates at a non-zero value contrary to the convergence in Figure~\ref{fig:ex3b_lambda}.

From Figure~\ref{fig:ex4_1} and~\ref{fig:ex4_2} we can see, that 
$\|\mathbf{u}^\infty - \mathbf{u}^\lambda_h\|_1$ converges to 
the constant ($\|\mathbf{u}^\infty - \mathbf{u}^\lambda\|_1 + ch^2\|\mathbf{u}\|_3$) as $\lambda \to \infty$
and $\|\mathbf{u}^\infty - \mathbf{u}^\lambda_h\|_1 \to 0$ (since, $\|\mathbf{u}^\infty - \mathbf{u}^\lambda\|_1 \to 0$)
as $\lambda \to \infty$ and $ch^2\|\mathbf{u}\|_3 \to 0$ as $h \to 0$.

\begin{figure}
  \begin{minipage}{\textwidth}
    \begin{subfigure}{0.45\textwidth}
      \centering
      \includegraphics[width = \textwidth]{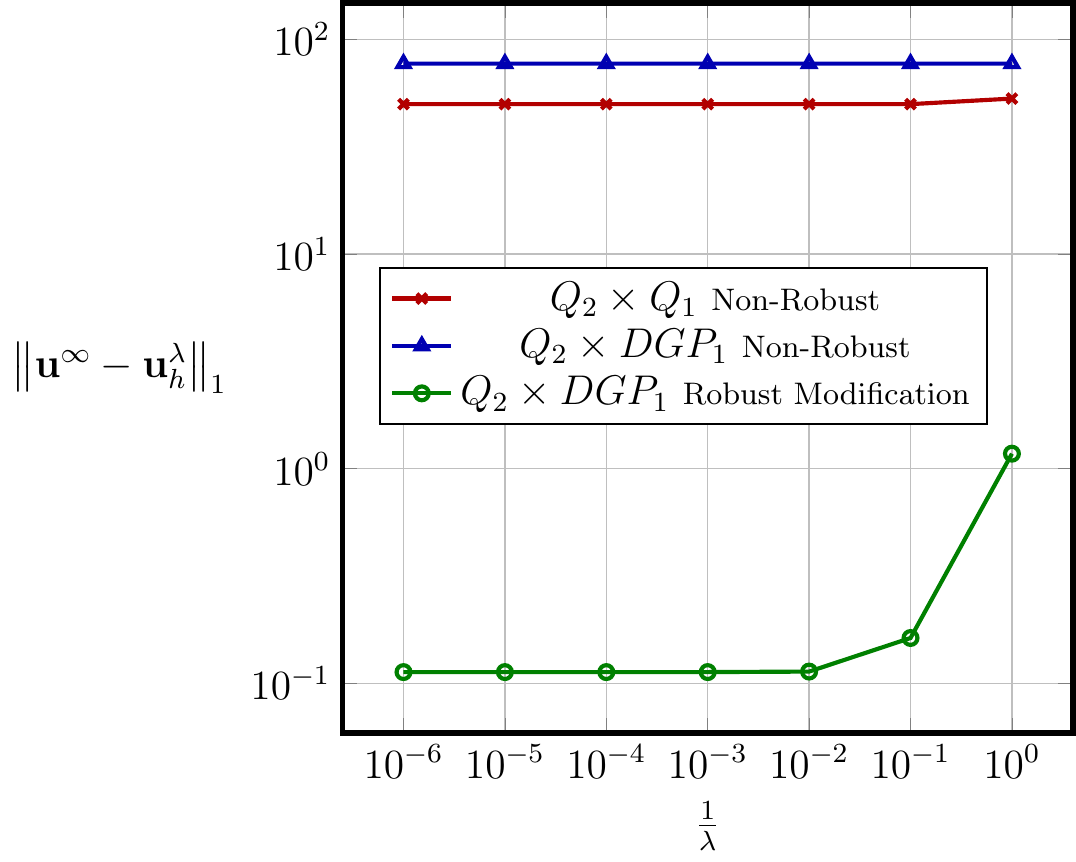}
      \caption{\small{$\|\mathbf{u}^\infty-\mathbf{u}_h^\lambda\|_1$  vs. $\frac{1}{\lambda}$ with $\mu = 10^{-5}$}}\label{fig:ex4_lambda}
    \end{subfigure}%
    \hspace*{\fill}
    \begin{subfigure}{0.45\textwidth}
      \centering
      \includegraphics[width = \textwidth]{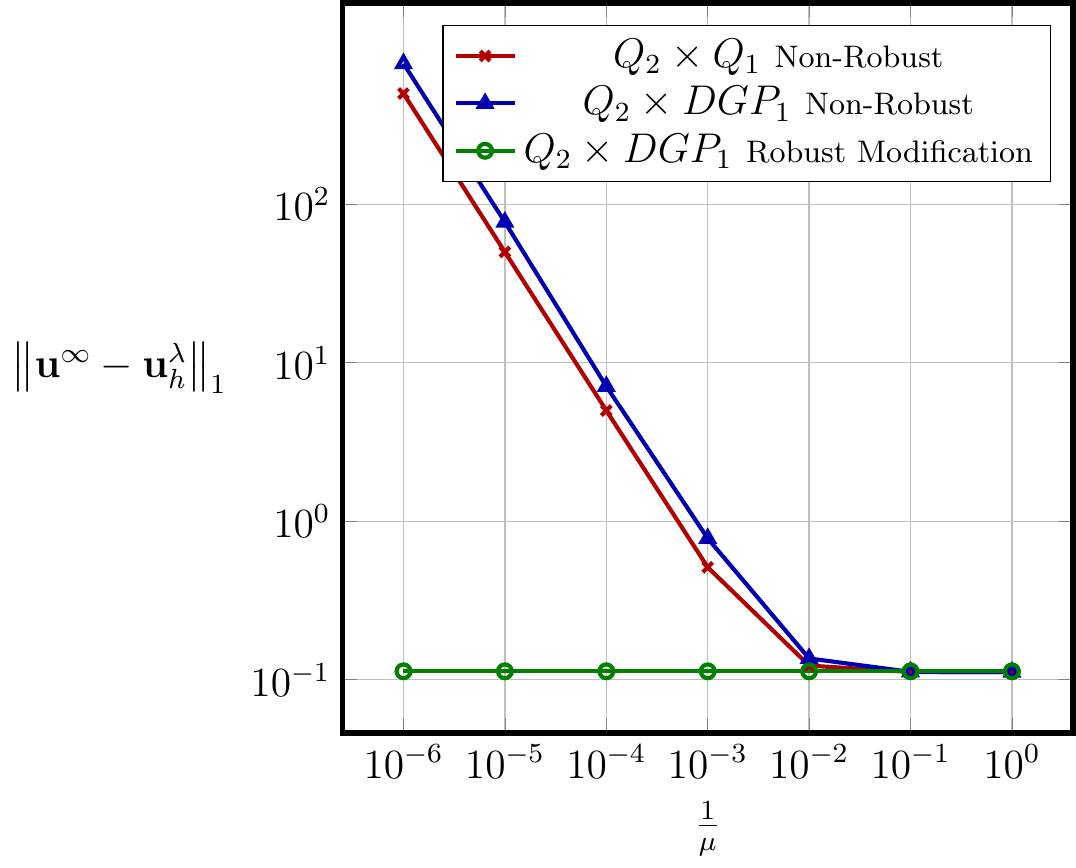}
      \caption{\small{$\|\mathbf{u}^\infty-\mathbf{u}_h^\lambda\|_1$ vs. $\frac{1}{\mu}$ with $\lambda = 10^{-5}$}} \label{fig:ex4_mu}
    \end{subfigure}
  \end{minipage}
  \caption{\small{Comparing displacement error in $H^1$ norm for 
      Example~\ref{ex:ex4} with and without gradient robust
      modification for $h = 2^{-3}$}}\label{fig:ex4}
\end{figure}

\begin{figure}[H]
  \begin{minipage}{\textwidth}
    \begin{subfigure}{0.45\textwidth}
      \centering
      \includegraphics[width = \textwidth]{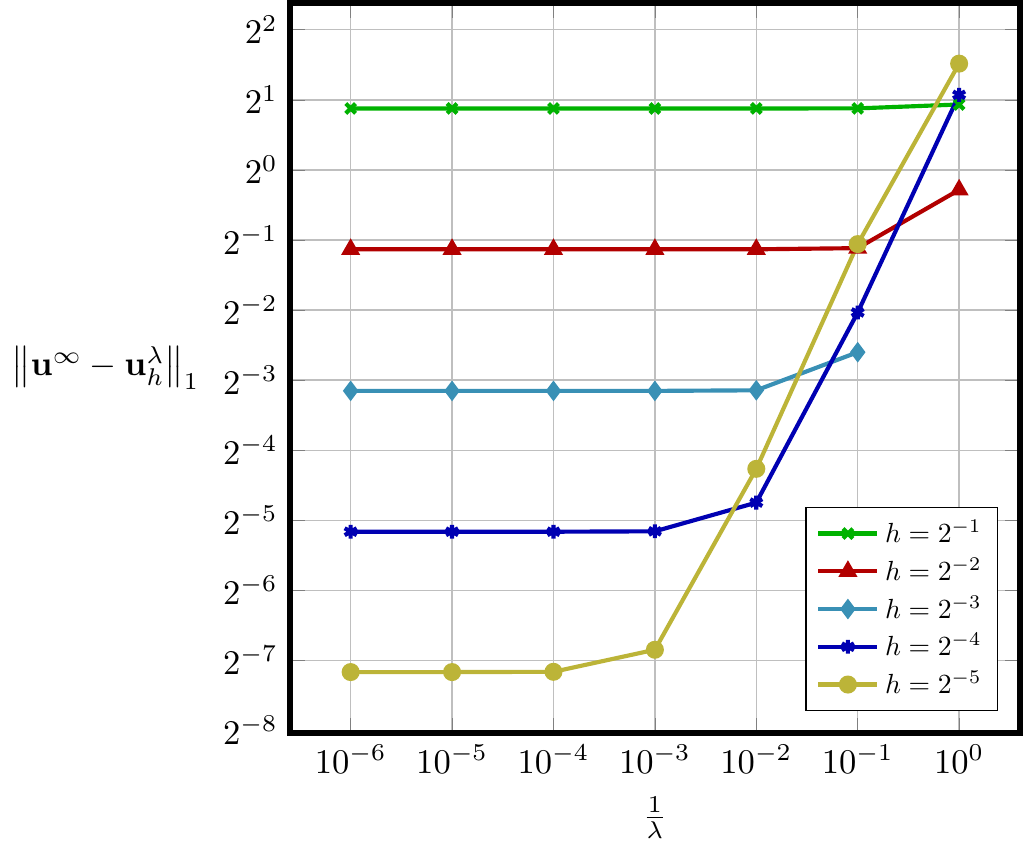}
      \caption{$\|\mathbf{u}^\infty - \mathbf{u}^\lambda_h\|_1$ vs. $\frac{1}{\lambda}$ } \label{fig:ex4_1}
    \end{subfigure}
    \hspace*{\fill}
    \begin{subfigure}{0.45\textwidth}
      \centering
      \includegraphics[width = \textwidth]{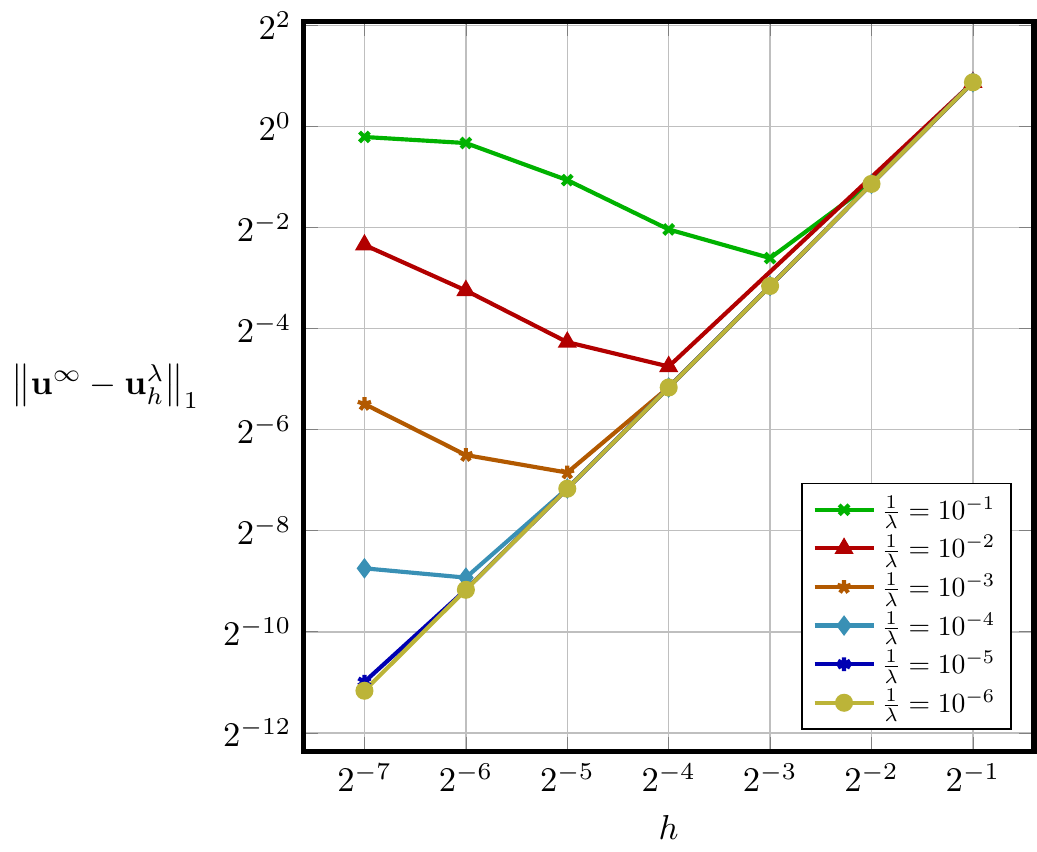}
      \caption{$\|\mathbf{u}^\infty - \mathbf{u}^\lambda_h\|_1$ vs. $h$} \label{fig:ex4_2}
    \end{subfigure}
  \end{minipage}
  \caption{Comparing displacement error in $H^1$ norm for the robust
    modification of Example~\ref{ex:ex4} for $\mu = 10^{-5}$}
  \label{fig:ex4a}
\end{figure}

\begin{example}
  Finally, we would like to compare our results with the thermo-elastic solids example given in~\cite[Section~6]{FuGuoshengLehrenfeldChristophLinkeStrechenbach:2020}.
  The gradient force $\mathbf{f}$ is given by a temperature $\theta$ as  
  \[ \mathbf{f} = -\left(2\mu + 3\lambda\right)\alpha \nabla \theta. \]
  The material used is a nearly incompressible hard rubber with Young's Modulus $E = 5 \times 10^7[\rm{Pa}]$,
  Poisson ratio $\nu = 0.4999$ and the thermal expansion coefficient $\alpha = 8 \times 10^{-5}[\rm{1/K}]$. Hence
  the Lam\'{e} parameters are $\lambda = 8.332 \times 10^{10}[\rm{Pa}]$ and $ \mu = 1.6667 \times 10^7[\rm{Pa}]$.
  We take the domain $\Omega = [0, L]^2$ with $L = 0.1[\rm{m}]$. The temperature field is obtained as the 
  solution to the stationary heat equation:
  \[-\nabla \cdot \gamma \nabla \theta = f,\]
  where $\gamma = 0.2[\rm{W/(m K)}]$ is the thermal conductivity coefficient and $f = 4 \times \rm{exp}(-40r^2)[\rm{W/m^3}]$ 
  is the heat source, with $r^2 = (x-0.5L)^2 + (y-0.5L)^2$. Homogeneous Dirichlet boundary conditions are applied 
  on both temperature and displacement. It is important to note that
  $\theta \in H^1(\Omega)$ and
  thus $f \in L^2(\Omega)$. For numerical computation, we
  additionally solve the temperature equation by a standard
  $H^1$-conforming finite element discretization. Hence, the finite
  element spaces now consist of three components, the first two
  denote the displacement and pressure discretization as before.
  The third element, always $Q_2$, is used to solve the equation 
  for the temperature $\theta$.

  In Figure~\ref{fig:Q2DGP1_Interpol}, we can see that we achieve a
  well represented solution for the displacement with only $64$
  elements using a gradient robust method, and the magnitude is
  already captured with only $16$ elements. In comparison, the non
  gradient robust methods require $256$ and $1024$ elements,
  respectively, to get a solution of similar shape and magnitude, see
  Figures~\ref{fig:Q2Q1_NoInterpol} and~\ref{fig:Q2DGP1_NoInterpol}.
\end{example}

\begin{figure}[H]
  \begin{center}
    \begin{subfigure}{0.32\textwidth}
      \centering
      \includegraphics[width = \textwidth]{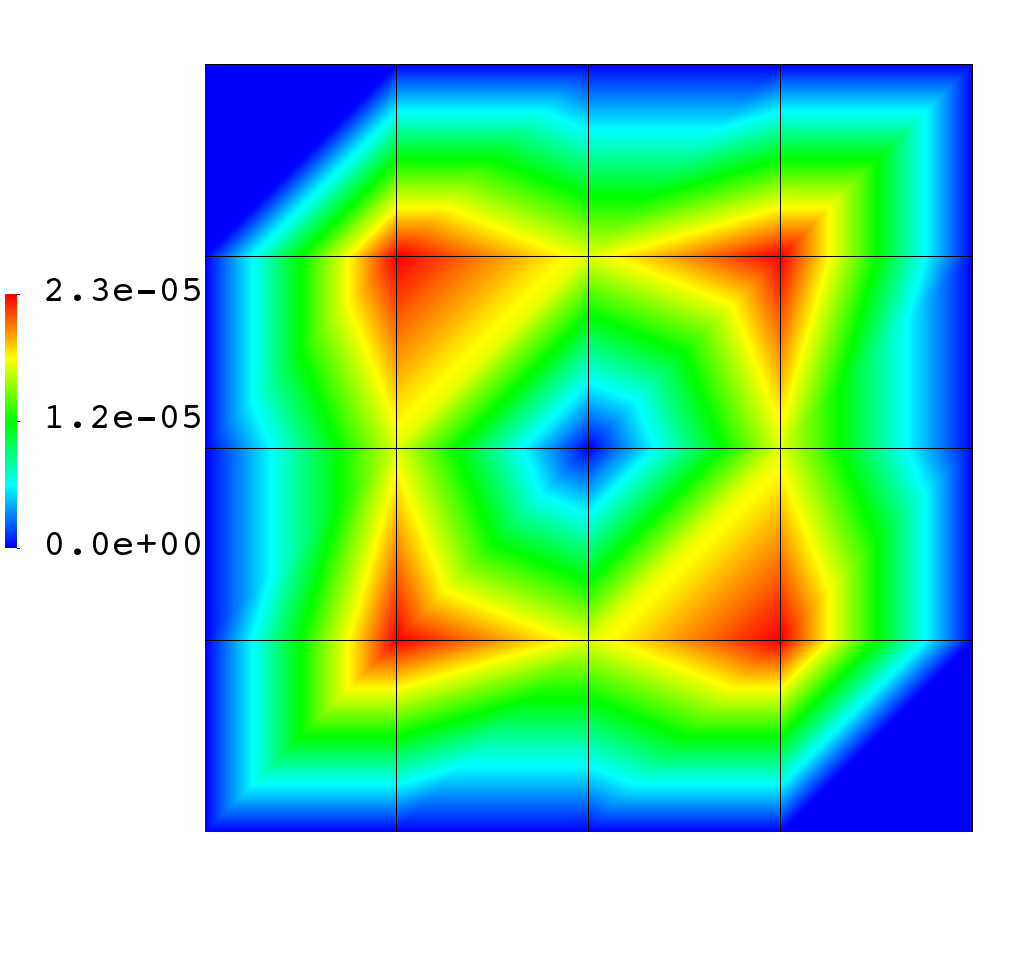}
      \caption{$16$ elements} \label{fig:291Interpol}
    \end{subfigure}
    \begin{subfigure}{0.32\textwidth}
      \centering
      \includegraphics[width=\textwidth]{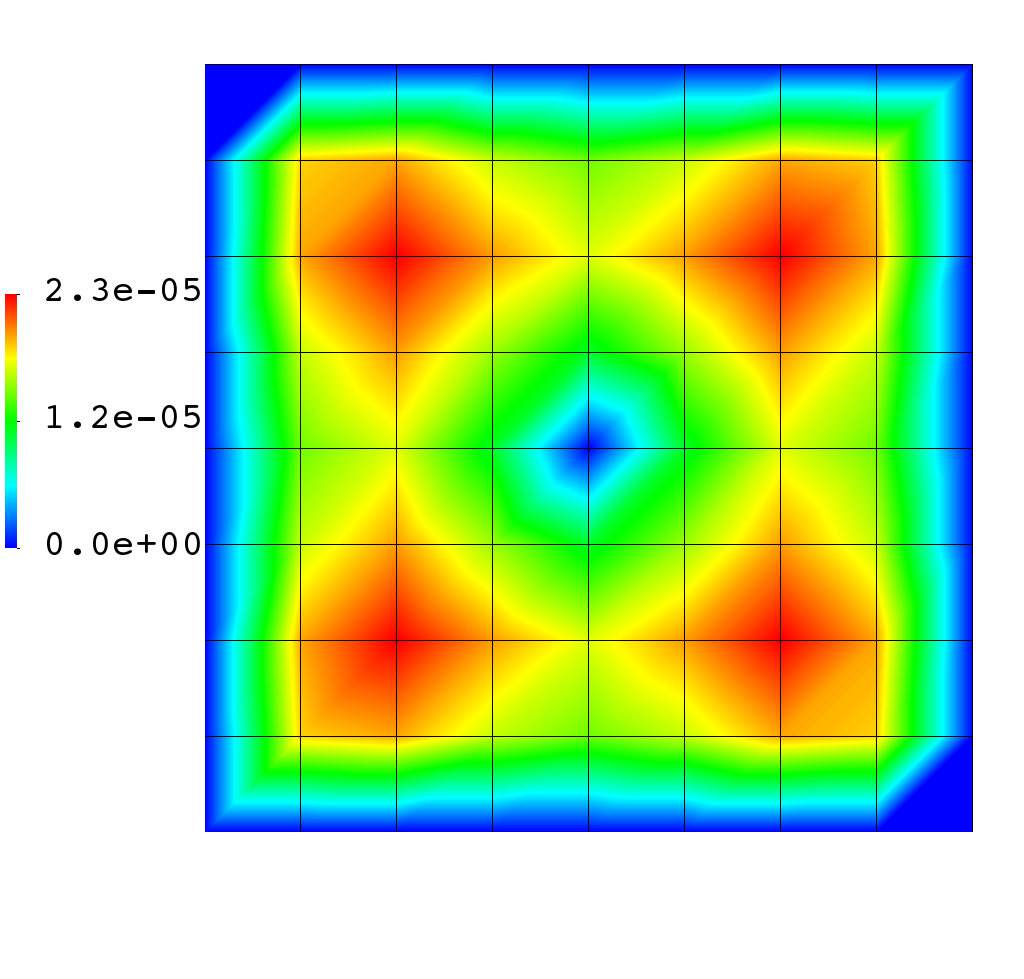}
      \caption{$64$ elements} \label{fig:1059Interpol}
    \end{subfigure}
  \end{center}
  \caption{Displacement vector for different number of elements with $Q_2 \times DGP_1\times Q_2$ 
    with BDM Interpolation} \label{fig:Q2DGP1_Interpol}
\end{figure}

\begin{figure}[H] 
  \begin{minipage}{\textwidth}
    \begin{subfigure}{0.32\textwidth}
      \centering
      \includegraphics[width=\textwidth]{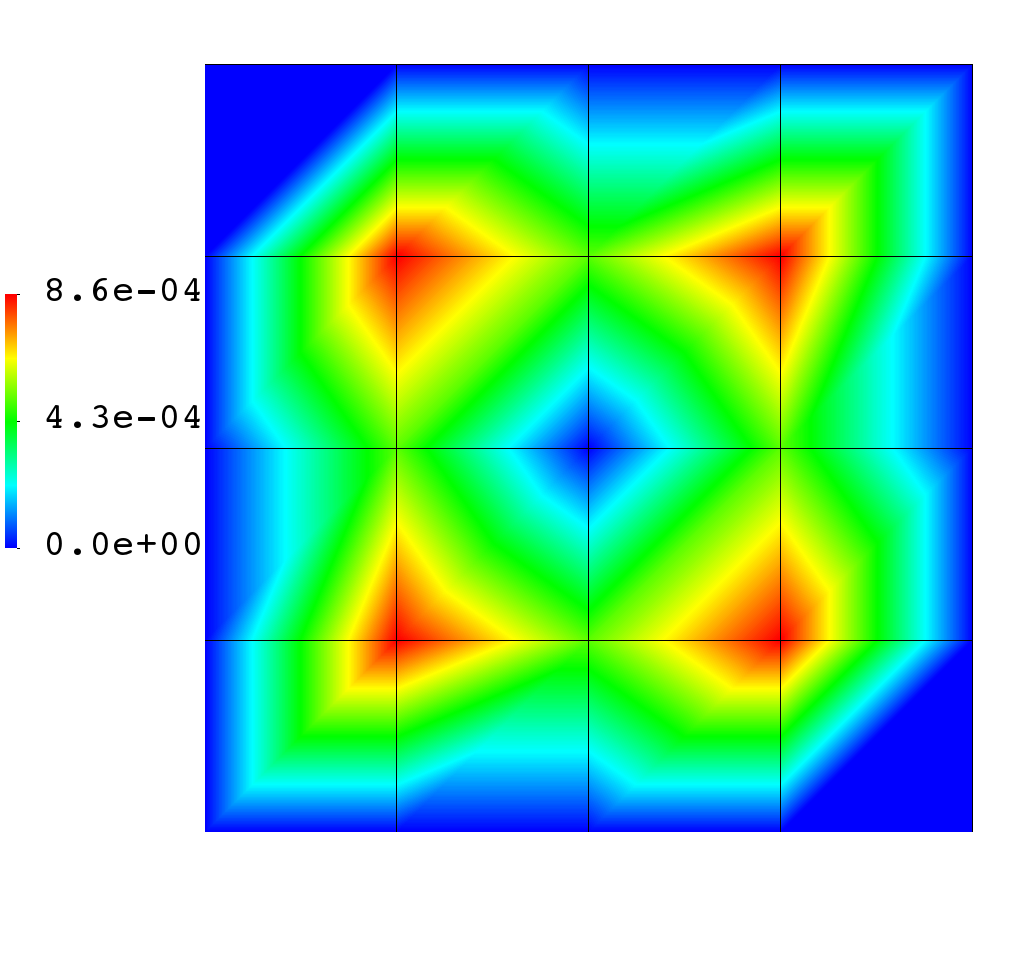}
      \caption{$16$ elements} \label{fig:268}
    \end{subfigure}
    \begin{subfigure}{0.32\textwidth}
      \centering
      \includegraphics[width=\textwidth]{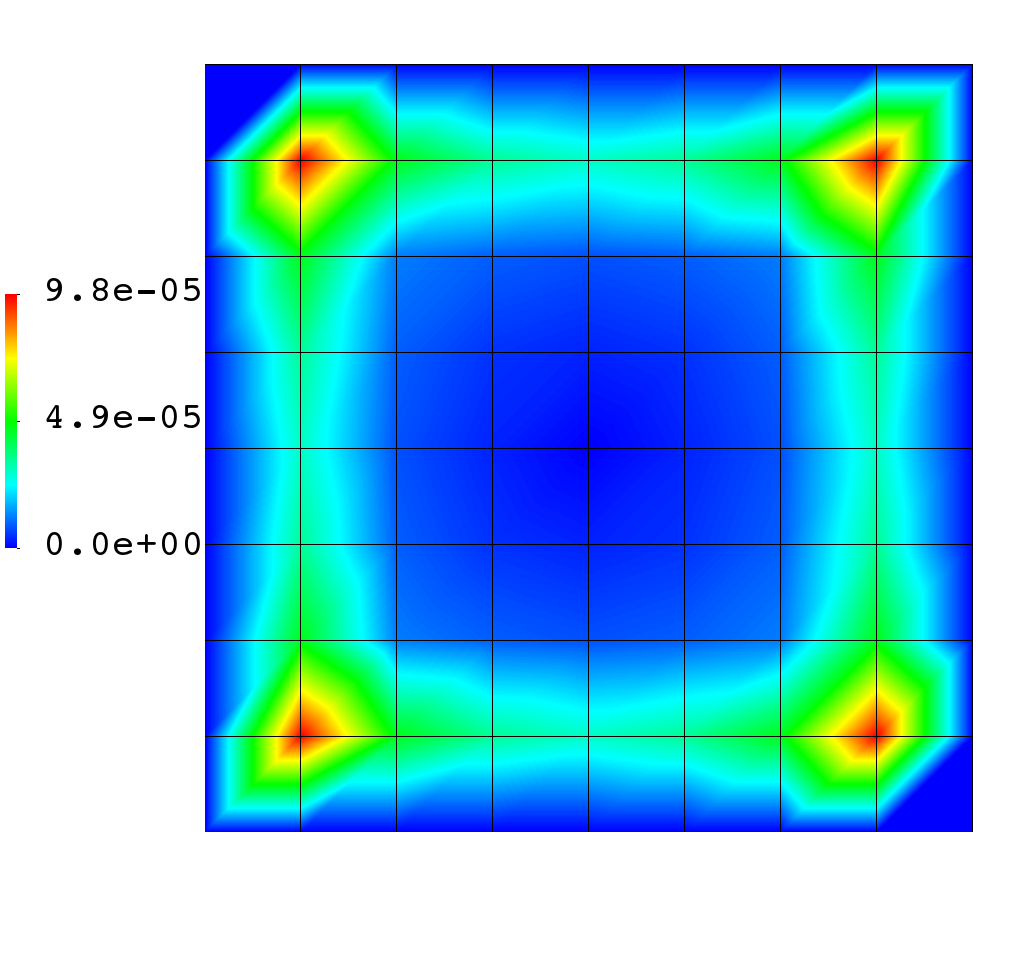}
      \caption{$64$ elements} \label{fig:948}
    \end{subfigure}
    \begin{subfigure}{0.32\textwidth}
      \centering
      \includegraphics[width=\textwidth]{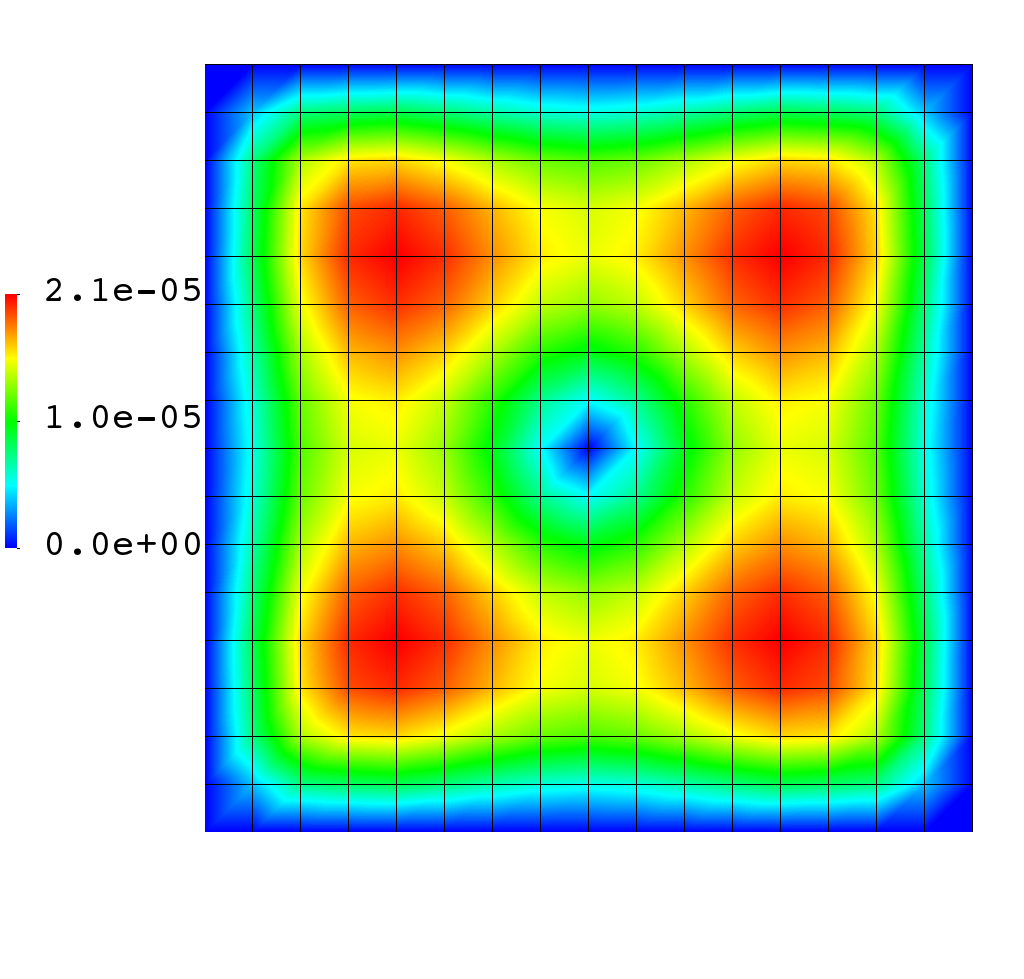}
      \caption{$256$ elements} \label{fig:3556}
    \end{subfigure}
  \end{minipage}
  \caption{Displacement vector for different number of elements with $Q_2 \times Q_1\times Q_2$}
  \label{fig:Q2Q1_NoInterpol}
\end{figure}

\begin{figure}[H]
  \begin{minipage}{\textwidth}
    \begin{subfigure}{0.32\textwidth}
      \centering
      \includegraphics[width=\textwidth]{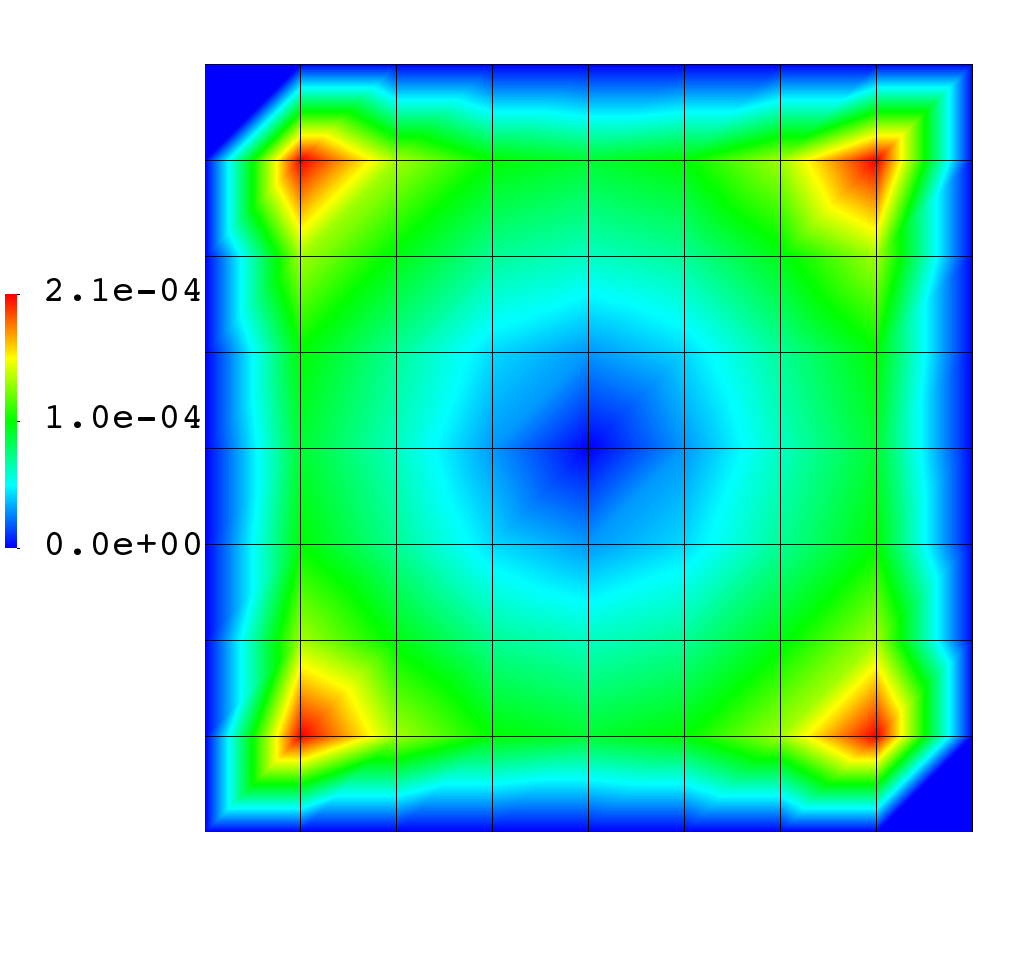}
      \caption{$64$ elements} \label{fig:1059NoInterpol}
    \end{subfigure}
    \begin{subfigure}{0.32\textwidth}
      \centering
      \includegraphics[width=\textwidth]{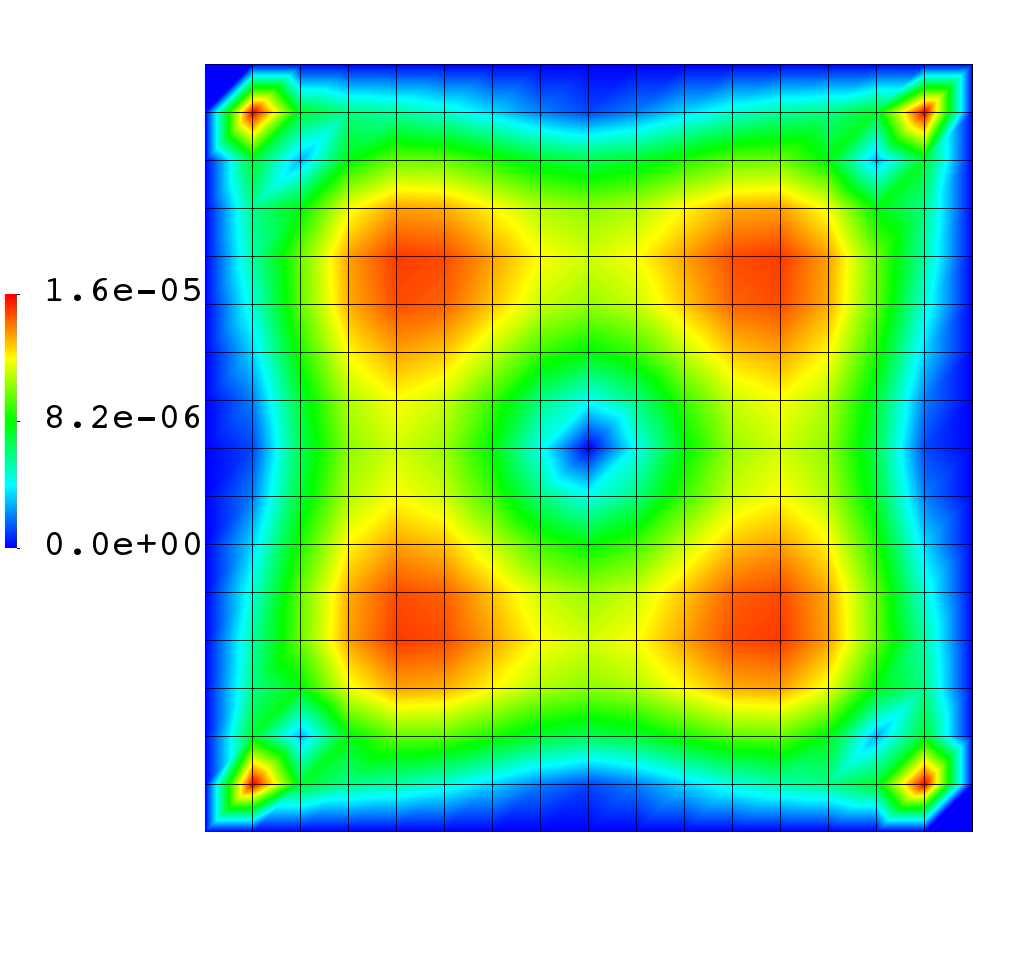}
      \caption{$256$ elements} \label{fig:4035NoInterpol}
    \end{subfigure}
    \begin{subfigure}{0.32\textwidth}
      \centering
      \includegraphics[width=\textwidth]{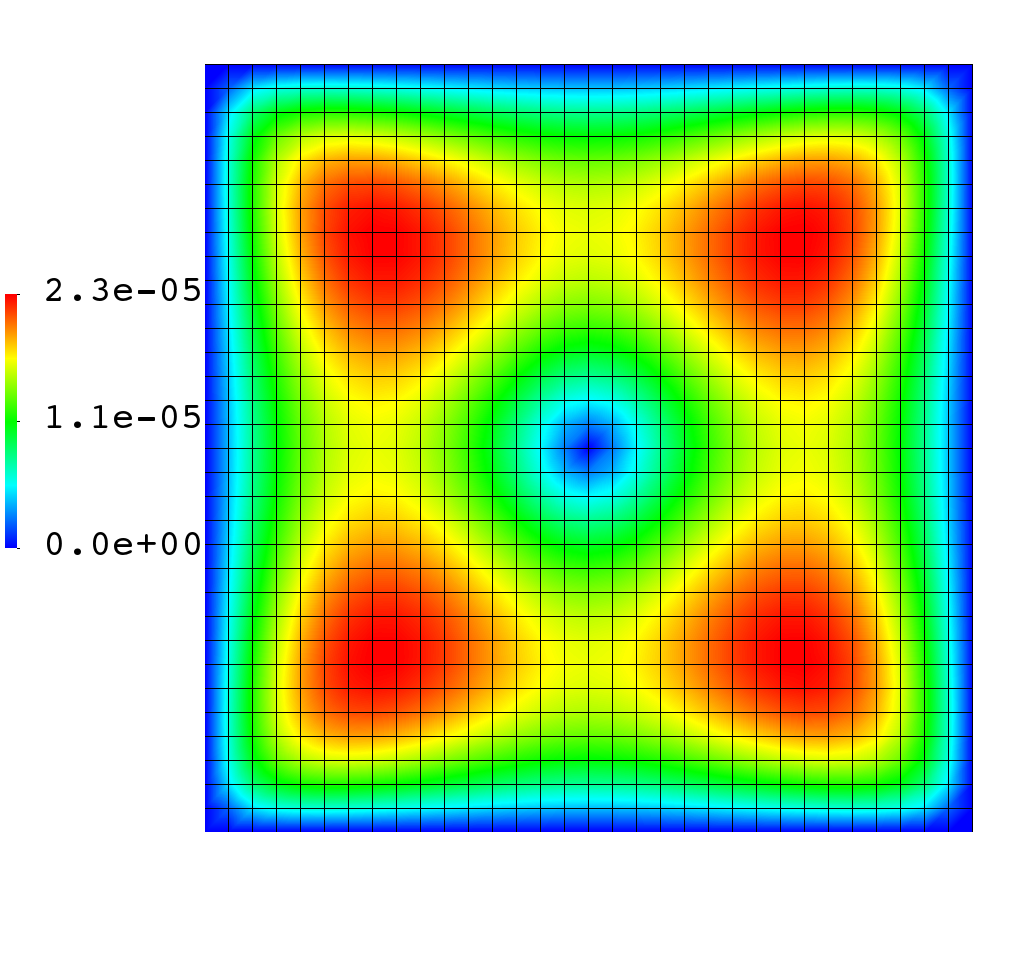}
      \caption{$1024$ elements} \label{fig:15747NoInterpol}
    \end{subfigure}
  \end{minipage}
  \caption{Displacement vector for different number of elements with $Q_2 \times DGP_1\times Q_2$}\label{fig:Q2DGP1_NoInterpol}
\end{figure}

\bibliographystyle{abbrv}
\bibliography{ref}

\end{document}